\documentclass[a4paper,12pt]{article}
\usepackage[left=2cm,right=2cm, top=2cm,bottom=3cm,bindingoffset=0cm]{geometry}
\usepackage{cmap}
\usepackage[T2A]{fontenc}
\usepackage[cp1251]{inputenc}
\usepackage{verbatim}
\usepackage{amsmath}
\usepackage{amsthm}
\usepackage{amssymb}
\usepackage{delarray}
\usepackage{cite}
\usepackage{hyperref}
\usepackage{mathrsfs}
\usepackage{tikz}
\usetikzlibrary{patterns}
\usepackage{caption}
\newcommand{\la}{\lambda}
\newcommand{\ga}{\gamma}

\DeclareCaptionLabelSeparator{dot}{. }
\theoremstyle{plain}

\numberwithin{equation}{section}

\newtheorem{thm}{Theorem}[section]
\newtheorem{lem}[thm]{Lemma}
\newtheorem{prop}[thm]{Proposition}
\newtheorem{cor}[thm]{Corollary}

\theoremstyle{definition}

\newtheorem{ip}[thm]{Inverse Problem}

\newtheorem{notation}{Notation}

\theoremstyle{remark}

 \allowdisplaybreaks

\begin{document}
\begin{center}
{\large\bf Inverse Sturm-Liouville problem with polynomials in the boundary \\[0.2cm] condition and multiple eigenvalues}
\\[0.2cm]
{\bf Egor E. Chitorkin, Natalia P. Bondarenko} \\[0.2cm]
\end{center}

\vspace{0.5cm}

{\bf Abstract.} In this paper, the inverse Sturm-Liouville problem with distribution potential and with polynomials of the spectral parameter in one of the boundary conditions is considered. We for the first time prove local solvability and stability of this inverse problem in the general non-self-adjoint case, taking possible splitting of multiple eigenvalues into account. The proof is based on the reduction of the nonlinear inverse problem to a linear equation in the Banach space of continuous functions on some circular contour. Moreover, we introduce the generalized Cauchy data, which will be useful for investigation of partial inverse Sturm-Liouville problems with polynomials in the boundary conditions. Local solvability and stability of recovering the potential and the polynomials from the generalized Cauchy data are obtained. Thus, the results of this paper include the first existence theorems for solution of the inverse Sturm-Liouville problems with polynomial dependence on the spectral parameter in the boundary conditions in the case of multiple eigenvalues. In addition, our stability results can be used for justification of numerical methods.

\medskip


\medskip


\vspace{1cm}

\section{Introduction} \label{sec:intro}

This paper deals with the Sturm-Liouville problem
\begin{gather} \label{eqv}
-y'' + q(x) y = \lambda y, \quad x \in (0, \pi), \\ \label{bc}
y^{[1]}(0) = 0, \quad r_1(\la)y^{[1]}(\pi) + r_2(\la)y(\pi) = 0,
\end{gather}
where $q(x)$ is a complex-valued distribution potential of the class $W_2^{-1}(0,\pi)$, that is, $q = \sigma'$, $\sigma \in L_2(0,\pi)$,
$\la$ is the spectral parameter, $r_1(\la)$ and $r_2(\la)$ are relatively prime polynomials, $y^{[1]} := y' - \sigma(x) y$ is the so-called quasi-derivative. 

We study the inverse problem that consists in the recovery of $\sigma(x)$, $r_1(\lambda)$, and $r_2(\lambda)$ from the spectral data.
This paper aims to prove local solvability and stability of the inverse problem in the general case of complex-valued potential and multiple eigenvalues.

% historical background

Inverse problems of the spectral analysis consist in recovering a specific differential operator, such as the Sturm-Liouvillle operator, from information about its spectrum. Such problems find their applications in different fields of science and engineering, particularly, in quantum and classical mechanics, geophysics, electronics, chemistry, nanotechnology.

Inverse spectral theory for the Sturm-Liouville operators with constant coefficients in the boundary conditions has been comprehensively developed (see the monographs \cite{Mar77, Lev84, FY01, Krav20}). There are also numerous studies on eigenvalue problems with polynomials which depend on the spectral parameter in the boundary conditions. Such problems arise in practical applications such as investigation of diffusion processes or of electrical circuit problems (see \cite{Ful77, Ful80} and references therein).

Inverse Sturm-Liouville problems with polynomials of the spectral parameter in the boundary conditions have been investigated in \cite{Chu01, BindBr021, BindBr022, ChFr, FrYu, FrYu12, YangWei18, Gul19, Gul20-ams, Gul23, Chit}. Some of these studies focus on self-adjoint problems that contain rational Herglotz-Nevanlinna functions in the boundary conditions (see \cite{BindBr021, BindBr022, YangWei18, Gul19, Gul20-ams}). Obviously, such problems can be reduced to problems with polynomial dependence on the spectral parameter. In contrast to the previously mentioned studies, this paper focuses on the general non-self-adjoint case. A constructive solution for the inverse non-self-adjoint Sturm-Liouville problem with boundary conditions containing polynomials of the spectral parameter has been developed by Freiling and Yurko \cite{FrYu, FrYu12} using the method of spectral mappings. However, Freiling and Yurko considered only regular (integrable) potentials.

In recent years, spectral analysis of differential operators with singular coefficients being generalized functions (distributions) has garnered significant interest among mathematicians (ss \cite{Gul19, SavShkal03, Hry03, Sav10, FrIgYu, Dj, BondTamkang, ChitBond} and other studies). Some properties of the spectrum and solutions of differential equations with singular coefficients have been explored in \cite{SavShkal03, Hry03, Sav10}. The most comprehensive results in the inverse problem theory have been obtained for the Sturm-Liouville operators with singular potentials (see \cite{SavShkal03, Hry03, Sav10, FrIgYu, Dj, BondTamkang, ChitBond}). Specifically, Hryniv and Mykytyuk \cite{Hry03} extended several classical results to the case of potentials from $W^{-1}_2$. Savchuk and Shkalikov proved the uniform stability of the inverse Sturm-Liouville problems with potentials of $W^{\alpha}_2$, where $\alpha > -1$ (see \cite{Sav10}). The method of spectral mappings was transferred to the case of distribution potentials in \cite{FrIgYu, BondTamkang}. In \cite{ChitBond}, the ideas of this method have been developed by the authors for constructive solution of an inverse problem with polynomials in the boundary conditions.

Investigation of local solvability and stability for various classes of inverse Sturm-Liouville problems was carried out in \cite{Borg46, BSY13, BondButTr17, BK19, Bond20, BondGaidel, GMXA23, ChitBondStab} and other studies. Local solvability is a crucial property of inverse problems, particularly in the cases when proving global solvability is challenging. Stability plays a vital role in justifying numerical methods. Guliyev \cite{Gul19} made a significant contribution to investigation of the global solvability for inverse problems with polynomials in boundary conditions in the self-adjoint case. However, there was a lack of such studies for the non-self-adjoint case. In \cite{ChitBondStab}, the authors have proved local solvability and stability for the case of simple eigenvalues. However, in general, eigenvalues of the problem \eqref{eqv}--\eqref{bc} can be multiple. Moreover, multiple eigenvalues can split into smaller groups under perturbations of the spectrum, and such splitting causes significant difficulties in investigation of local solvability and stability. For the standard Sturm-Liouville problem in the non-self-adjoint case, these issues were considered in \cite{BSY13} but only for multiple eigenvalues remaining unperturbed. Recently, Buterin and Kuznetsova \cite{BK19} generalized the classical result of Borg \cite{Borg46} for the two spectra inverse problem to the non-self-adjoint case, taking splitting of multiple eigenvalues into account. Local solvability and stability for the inverse problem by the generalized spectral data (eigenvalues and generalized weight numbers) have been investigated by Bondarenko \cite{Bond20} for the Sturm-Liouville operators and by Bondarenko and Gaidel \cite{BondGaidel} for quadratic differential pencils. The approach of \cite{Bond20, BondGaidel} relies on the fact that the number of multiple eigenvalues is finite. Therefore, it is convenient to consider some rational function on a circular contour in the complex $\la$-plane instead of a finite part of the spectral data.

The goal of this paper is to prove local solvability and stability for the inverse Sturm-Liouville problem with polynomials in the boundary condition in the most general case, taking splitting of multiple eigenvalues into account. Our approach is based on the constructive solution of the inverse problem by the method of spectral mappings from our previous study \cite{ChitBond}, which develops the ideas of \cite{FY01, FrYu, BondTamkang}. This method reduces an inverse problem to a linear equation in a suitable Banach space. In this paper, we first consider small perturbations of a finite part of the spectral data, which can contain multiplicities. Then, we get the main linear equation in the space of continuous functions on some contour. Investigation of properties for the main equation and its solution leads to local solvability and stability of the inverse problem. An important role in our technique is played by reconstruction formulas for the coefficient $\sigma(x)$ and especially for the polynomials $r_1(\la)$ and $r_2(\la)$. For some special cases, such reconstruction formulas were derived in \cite{ChitBond, ChitBondStab}. In this paper, we obtain reconstruction formulas for the case of multiple eigenvalues, which requires a new complicated technique. Finally, we combine our results for a finite perturbation of the spectral data with the previous results of \cite{ChitBondStab} for simple eigenvalues, and so arrive at local solvability and stability in the general case.

In addition, we consider the inverse problem by the so-called generalized Cauchy data. For the classical Sturm-Liouville equation, Rundell and Sacks \cite{RS92} discovered that, for numerical reconstruction of the classical Sturm-Liouville operator, it is convenient to use the Cauchy data of the corresponding wave equation, which are equivalent to the Weyl function or the two spectra. Later on, those reconstruction techniques were developed for the inverse transmission eigenvalue problem (see \cite{MPS94, MSS97}).
Local solvability and stability for the inverse problem by the Cauchy data has been proved in \cite{BondAn}. That result and its generalizations were applied to investigation of partial inverse problems \cite{KuzCauchy, Bond23}, the inverse transmission eigenvalue problem \cite{BCK20}, and the Regge-type problem \cite{XuBondStab}. In this paper, we generalize the notion of Cauchy data to the case of polynomials in the boundary condition (see Section~\ref{sec:main} for details) and prove local solvability and stability for recovering $\sigma(x)$, $r_1(\la)$, and $r_2(\la)$ from the generalized Cauchy data. For the proof, we use the generalized Cauchy data to construct the spectral data (eigenvalues and weight numbers) and then apply our main result on local solvability and stability for the inverse spectral problem. In the future, we plan to use the results of this paper for studying partial inverse problems and the Sturm-Liouville problems with analytical dependence on the spectral parameter in one of the boundary conditions, which were considered in \cite{Chit}.

It is worth explaining why we consider the polynomial dependence on the spectral parameter in the boundary conditions \eqref{bc} only at the right endpoint. Inverse problems with polynomials in the both boundary conditions can be studied (see, e.g., \cite{FrYu}) but then the polynomials at one of the endpoints have to be given a priori. In \cite{ChitBond}, the authors have shown how the inverse problem with polynomials in the both boundary conditions can be transformed to the inverse problem with polynomials at a single endpoint. Therefore, we carry out our further research for the latter case. Furthermore, we consider a singular potential $q(x)$ of the class $W_2^{-1}(0,\pi)$ in order to simplify the technical work with spectral data asymptotics and with the generalized Cauchy data. Analogous results can be obtained for a regular potential $q \in L_2(0,\pi)$ but this requires additional complicated calculations.

The paper is organized as follows. In Section~\ref{sec:main}, the main results and the proof strategy are provided. Section~\ref{sec:prelim} contains some preliminaries. In Section~\ref{sec:maineq}, we present the main equation of the inverse problem and obtain some features of its solution. In Section~\ref{sec:rec}, we derive reconstruction formulas for the polynomials in the case of multiple eigenvalues. In Section~\ref{sec:solvstab}, the theorem on local solvability and stability of the inverse spectral problem is proved. Section~\ref{sec:cauchy} is devoted to local solvability and stability of the inverse problem by the generalized Cauchy data.

\section{Main results and proof strategy} \label{sec:main}

Denote by $L = L(\sigma, r_1, r_2)$ the boundary value problem \eqref{eqv}-\eqref{bc}. 
Due to the standard regularization approach of \cite{SavShkal03},
we understand equation \eqref{eqv} in the following equivalent sense:
\begin{equation} \notag
-(y^{[1]})' - \sigma(x)y^{[1]} - \sigma^2(x)y = \la y, \quad x \in (0,\pi),
\end{equation}
where $y, y^{[1]} \in AC[0,\pi]$.
Obviously, the polynomials of the boundary conditions \eqref{bc} can be expressed in the form
$$
r_1(\la) = 	\sum \limits_{n=0}^{K_1} c_n \la^n, \quad  r_2(\la) = \sum \limits_{n=0}^{K_2} d_n \la^n, \quad K_1, K_2 \ge 0.
$$

We will write $(r_1, r_2)\in\mathcal R_p $ if $K_1 = K_2 = p$ and $c_p = 1$. Note that the higher coefficients of $r_2(\lambda)$ can be equal to zero. Throughout this paper, we assume that $(r_1, r_2)\in\mathcal R_p$. The opposite case $K_2 < K_1$ can be considered similarly.

Let us define $S(x, \la)$ and $\varphi(x, \la)$ as the solutions of equation~\eqref{eqv} under the initial conditions:
\begin{equation} \label{initphi}
S(0, \la) = 0, \quad S^{[1]}(0, \la) = 1, \quad \varphi(0, \la) = 1, \quad \varphi^{[1]}(0, \la) = 0.
\end{equation}
Clearly, the functions $S(x, \la)$, $S^{[1]}(x, \la)$, $\varphi(x, \la)$ and $\varphi^{[1]}(x, \la)$ are entire analytic in $\la$ for each fixed $x \in [0,\pi]$.

Let us also introduce the auxiliary boundary value problem $L_0$ for equation~\eqref{eqv} with the boundary conditions
$$
y(0) = 0, \quad r_1(\la)y^{[1]}(\pi) + r_2(\la)y(\pi) = 0.
$$

It can be shown by simple calculations that the spectra of the problems $L$ and $L_0$ coincide with the zeros of the characteristic functions
\begin{gather}\label{charfun}
\Delta_1(\la) = r_1(\la)\varphi^{[1]}(\pi, \la) + r_2(\la)\varphi(\pi, \la),\\ \notag
\Delta_0(\la) = r_1(\la)S^{[1]}(\pi, \la) + r_2(\la)S(\pi, \la), 
\end{gather}
respectively.

Define the Weyl function of the boundary value problem $L$:
\begin{equation} \label{fracD01}
M(\la) = -\dfrac{\Delta_0(\la)}{\Delta_1(\la)}.
\end{equation}
From this representation, we can see that $M(\la)$ is uniquely determined by the spectra of the problems $L$ and $L_0$.

It has been proved in \cite{ChitBond} that the spectrum of the boundary value problem $L$ is a countable set of eigenvalues. They can be numbered as $\{\la_n\}_{n \ge 1}$ basing on their asymptotic behavior:
\begin{gather}\label{la_asymp}
\rho_n = \sqrt{\la_n} = n - p - 1 + \varkappa_n, \quad \{ \varkappa_n\} \in l_2.
\end{gather}

The asymptotics \eqref{la_asymp} implies that the eigenvalues are simple for sufficiently large indices $n$. However, a finite number of eigenvalues in general can be multiple.
Let $I = \{1\}\cup\{n > 1 \colon \la_n \neq \la_{n-1}\}$ be the set of the indices of all the distinct eigenvalues, and let $m_n$, $n \in I$, be their multiplicities. 
Then, the Weyl function $M(\la)$ is meromorphic in $\la$ and all its singularities are the poles $\{ \la_n \}_{n \in I}$. Due to the results of \cite{ChitBond}, the following representation holds:
\begin{gather} \label{double_sum}
M(\la) = \sum_{n \in I} \sum_{k=0}^{m_k-1} \frac{\alpha_{n+k}}{(\la - \la_n)^{k+1}},
\end{gather}
where the coefficients $\{ \alpha_n \}_{n \ge 1}$ are called the weight numbers.
We refer to the numbers $\{\la_n, \alpha_n\}_{n \ge 1}$ as \textit{the spectral data} of the problem $L$ and study the following inverse problem.

\begin{ip}\label{ip:main3}
	Given the spectral data $\{ \la_n, \alpha_n \}_{n \ge 1}$, find $\sigma(x)$, $r_1(\la)$, and $r_2(\la)$.
\end{ip}

Obviously, Inverse Problem~\ref{ip:main3} is equivalent to the reconstruction of $\sigma$, $r_1$, and $r_2$ from the Weyl function $M(\la)$ or from the two spectra of the boundary value problems $L$ and $L_0$. The uniqueness for solution of Inverse Problem~\ref{ip:main3} has been proved in \cite{ChitBond}.

In addition to $L$, let us consider another boundary value problem $\tilde L = L(\tilde \sigma, \tilde r_1, \tilde r_2)$ of the same form but with different coefficients. We agree that, if a symbol $\ga$ denotes an object related to $L$, then the symbol $\tilde \ga$ with a tilde will denote the corresponding object related to $\tilde L$, and $\hat \ga = \ga - \tilde \ga$.

\begin{prop}[\cite{ChitBond}] \label{prop:uniq}
If $M(\la) = \tilde M(\la)$, then $\sigma(x) = \tilde\sigma(x)$ a.e. on $(0, \pi)$ and $r_j (\la) \equiv \tilde r_j (\la)$, j = 1, 2.
\end{prop}

Since the Weyl function is uniquely determined by the spectral data, Proposition~\ref{prop:uniq} implies the uniqueness for solution of Inverse Problem~\ref{ip:main3}.

\begin{notation} \label{not:GammaN}
Let $N$ be such an index that $\tilde{m}_n = 1$ for $n \ge N$ 
and that the eigenvalues $\{ \tilde \la_n \}_{n = 1}^{N-1}$ of $\tilde L$ lie inside the contour 
$$
\Gamma_N = \Bigg\{ \la \in \mathbb C \colon |\la| = \bigg(N-p-\dfrac{3}{2}\bigg)^2 \Bigg\}.
$$
\end{notation}

Introduce the function
\begin{gather} \label{mn_func}
M_N(\la) = \sum\limits_{n \in I, n < N} \sum\limits_{k=0}^{m_k-1} \dfrac{\alpha_{n+k}}{(\la-\la_n)^{k+1}}.
\end{gather}

Put $\rho_n := \sqrt{\la_n}$, $\arg \rho_n \in \left( -\tfrac{\pi}{2}, \tfrac{\pi}{2} \right]$, $n \ge 1$.

The first of our main results is the following theorem on local solvability and stability of Inverse Problem~\ref{ip:main3}.

\begin{thm} \label{stability_thm}
Let $\tilde L = L(\tilde \sigma, \tilde r_1, \tilde r_2)$ be a fixed boundary value problem of form \eqref{eqv}-\eqref{bc} with $\tilde \sigma \in L_2(0,\pi)$, $(\tilde r_1, \tilde r_2) \in \mathcal R_p$, and let the index $N$ be defined as above. 
Then, there exists $\delta_0 > 0$ depending on problem $\tilde L$ such that, for any complex numbers $\{\la_n, \alpha_n\}_{n \ge 1}$ satisfying the condition
	\begin{equation} \label{estde}
	\delta := \max\Bigg\{\max\limits_{\la \in \Gamma_N}|\hat M_N(\la)|,\bigg(\sum\limits_{n=N}^\infty(|\tilde \rho_n - \rho_n|+|\tilde\alpha_n - \alpha_n|)^2\bigg)^{\frac{1}{2}}\Bigg\} \le \delta_0,
	\end{equation}
	where $\rho_n := \sqrt{\la_n}$, $\tilde \rho_n := \sqrt{\tilde \la_n}$,
	there exist a complex-valued function $\sigma(x) \in L_2(0, \pi)$ and polynomials $(r_1, r_2)\in\mathcal R_p$, which are the solution of Inverse Problem~\ref{ip:main3} for the data $\{\la_n, \alpha_n\}_{n \ge 1}$. Moreover,
	\begin{equation} \label{stab}
	\|\sigma - \tilde\sigma\|_{L_2(0, \pi)} \le C\delta, \quad |c_n - \tilde c_n| \le C\delta, \quad |d_n - \tilde d_n| \le C\delta, \quad n = \overline{0, p},
	\end{equation}
	where the constant $C$ depends only on $\tilde L$.
\end{thm}

Lemma A1 from \cite{BondGaidel} implies that, for the function $M_N(\la)$, which is built by formula \eqref{mn_func} and satisfies the condition \eqref{estde} for sufficiently small $\delta_0$, all its poles $\{\la_n\}_{n = 1}^{N-1}$ lie inside the contour $\Gamma_N$.

The proof of Theorem~\ref{stability_thm} is divided into two parts. In the first one, we perturb only finite part of the spectral data with multiple eigenvalues. We develop a constructive method for solution, which is based on the ideas of \cite{FrYu, BondTamkang, ChitBond}. In the case of a finite perturbation, the inverse problem is reduced to a linear equation in the Banach space $C(\Gamma_N)$ of the continuous functions on the contour $\Gamma_N$ with the norm 
$$
\| f \|_{C(\Gamma_N)} = \max_{\la \in \Gamma_N} |f(\la)|.
$$
The linear main equation is considered for each fixed $x \in [0,\pi]$ and has the form
$$
\tilde \varphi(x, \la) = (E + \tilde Q(x)) \varphi(x, \la),
$$
where $\tilde Q(x)$ is a linear integral operator in $C(\Gamma_N)$, which is rigorously defined in Section~\ref{sec:maineq}, and $E$ is the identity operator.

We show that, if $\| \hat M_N \|_{C(\Gamma_N)}$ is sufficiently small, then the operator $(E + \tilde Q(x))$ has a bounded inverse operator, so the main equation is uniquely solvable. Furthermore, we derive the formulas for the reconstruction of $\sigma(x)$, $r_1(\la)$, and $r_2(\la)$ by using the solution $\varphi(x, \la)$ of the main equation:

\begin{thm} \label{mainthm}
	Let $N$ be a fixed positive integer.
	Suppose that the spectral data of two boundary value problems $L$ and $\tilde L$ fulfill the conditions
	$\la_n = \tilde \la_n$, $\alpha_n = \tilde\alpha_n$ for $n \ge N$ and $\la_n, \tilde \la_n \in \mbox{int}\,\Gamma_N$ for $n < N$. Then, the following relations hold:
	\begin{align} \label{sigma_series}
	\sigma(x) = \tilde \sigma(x) - \frac{1}{2\pi i}\oint_{\Gamma_N}{\big(2\tilde\varphi(x, \mu)\varphi(x, \mu) - 1\big)\hat M_N(\mu)}d\mu,
	\end{align}
	\begin{align} \label{r1_series}
	r_1(\la) = \prod\limits_{k = 1}^{N-1} \dfrac{\la-\la_{k}}{\la-\tilde \la_{k}} \Bigg(\tilde r_1(\la) - \dfrac{1}{2\pi i}\oint_{\Gamma_N}{\dfrac{\tilde r_1(\la)\tilde\varphi^{[1]}(\pi, \mu) + \tilde r_2(\la)\tilde\varphi(\pi, \mu)}{\la - \mu}\varphi(\pi, \mu)\hat M_N(\mu)}d\mu \Bigg),
	\end{align}
	\begin{align}\notag 
	r_2(\la) &= \prod\limits_{k = 1}^{N-1} \dfrac{\la-\la_{k}}{\la-\tilde \la_{k}}\Bigg(\tilde r_2(\la) - \tilde r_1(\la)\frac{1}{2\pi i}\oint_{\Gamma_N}{\big(\tilde\varphi(\pi, \mu)\varphi(\pi, \mu) - 1\big)\hat M_N(\mu)}d\mu\\ \label{r2_series}
	&+\dfrac{1}{2\pi i}\oint_{\Gamma_N}{\dfrac{\tilde r_1(\la)\tilde\varphi^{[1]}(\pi, \mu) + \tilde r_2(\la)\tilde\varphi(\pi, \mu)}{\la - \mu}\varphi^{[1]}(\pi, \mu)\hat M_N(\mu)}d\mu \Bigg).
	\end{align}
\end{thm}

For the zero model problem with the coefficients $\tilde \sigma(x) \equiv 0$, $\tilde r_1(\la) = \la^{p}$, $\tilde r_2(\la) = 0$, the reconstruction formulae have been obtained in \cite{ChitBond}. Those formulae consist of two parts: the first one is specified by a finite part of spectrum with multiple eigenvalues and has an integral form, and the second one is specified by the infinite part of spectrum with simple eigenvalues and has a form of an infinite sum. Reconstruction formulae in the case of simple eigenvalues and an arbitrary model problem have been obtained in \cite{ChitBondStab}. Thus, the cases of the previous studies \cite{ChitBond} and \cite{ChitBondStab} are different from the situation considered in this paper. The derivation of the formulas \eqref{r1_series} and \eqref{r2_series} for the polynomials in our case is more technically complicated, so we provide the details in Section~\ref{sec:rec}.
Using the reconstruction formulas of Theorem~\ref{mainthm}, we prove the estimates \eqref{stab} and so conclude the proof of Theorem~\ref{thm_mult} for a finite perturbation of the spectral data.

In the second part of the proof, we perturb an infinite part of the spectral data with simple eigenvalues and apply Theorem 6.1 from \cite{ChitBondStab}. Finally, we combine the results of the two parts and arrive at the local solvability and stability of Inverse Problem~\ref{ip:main3} in the general case.

Our next result is concerned with the inverse problem by the generalized Cauchy data. Using the integral representations for $S(x, \la)$, $S^{[1]}(x, \la)$, $\varphi(x, \la)$, $\varphi^{[1]}(x, \la)$ from \cite{ChitBond} and direct calculations, which are based on ideas from \cite{Chit}, we prove the following lemma.

\begin{lem}
	The following representations hold
	\begin{align}\label{delta_0}
	\Delta_0(\la) & = \rho^{2p}\cos\rho\pi + \rho^{2p}\int\limits_{0}^{\pi} {J(t)\cos\rho t} \, dt + \sum\limits_{n=0}^{p-1}D_n\rho^{2n}, \\ \label{delta_1}
	\Delta_1(\la) & = - \rho^{2p+1}\sin\rho\pi + \rho^{2p+1}\int\limits_{0}^{\pi} {G(t)\sin\rho t} \, dt + \sum\limits_{n=0}^{p}C_n\rho^{2n},
	\end{align}
	where $\rho = \sqrt{\la}$, $J(t)$, $G(t) \in L_2(0, \pi)$ and $C_i$, $D_j$ for $i = \overline{0, p}$, $j = \overline{0,p-1}$ are complex numbers.
\end{lem}

We call the collection $\{ G(t), J(t), C_0, \dots, C_p, D_0, \dots, D_{p-1}\}$ the generalized Cauchy data for the problem $L$ because of the following reason. Consider the classical Sturm-Liouville equation \eqref{eqv} with an integrable potential $q$. Denote by $S(x, \la)$ the solution of \eqref{eqv} under the initial conditions $S(0, \la) = 0$, $S'(0, \la) = 1$. Then, this solution can be represented in terms of the transformation operator (see, e.g., \cite{Mar77}):
$$
S(x, \la) = \frac{\sin \rho x}{\rho} + \frac{1}{\rho} \int_0^{\pi} \mathscr K(x, t) \sin \rho t \, dt.
$$
Consequently, one can easily derive the following representations:
\begin{align} \label{reprS}
S(\pi, \la) & = \frac{\sin \rho \pi}{\rho} - \frac{\omega \cos \rho t}{\rho^2} + \frac{1}{\rho^2}\int_0^{\pi} \mathcal K(t) \cos \rho t \, dt, \\ \label{reprSp}
S'(\pi, \la) & = \cos \rho \pi + \frac{\omega \sin \rho \pi}{\rho} + \frac{1}{\rho} \int_0^{\pi} \mathcal N(t) \sin \rho t \, dt,
\end{align}
where $\omega = \frac{1}{2} \int_0^{\pi} q(x) \, dx$, $\mathcal K(t) := \mathscr K_t(\pi, t)$, $\mathcal N(t) := \mathscr K_x(\pi, t)$.

The pair of the functions $\{ \mathcal K(t), \mathcal N(t)\}$ is called the Cauchy data (see, e.g., \cite{BondAn}), because the Sturm-Liouville equation \eqref{eqv} with the Dirichlet boundary conditions $y(0) = y(\pi) = 0$ is closely related to the Cauchy problem
\begin{gather*}
u_{tt} - u_{xx} + q(x) u = 0, \quad 0 \le |t| \le x \le \pi, \\
u(\pi, t) = \mathscr K(\pi, t), \quad u_x(\pi, t) = \mathscr K_x(\pi, t), \quad -\pi \le t \le \pi,
\end{gather*}
where $\mathscr K(x, t) = -\mathscr K(x, t)$ for $t < 0$. The initial data $\{ {\mathscr K(\pi, t), \mathscr K_x(\pi, t)}\}$ of this Cauchy problem are related in an obvious way to the functions $\{ \mathcal K(t), \mathcal N(t)\}$.

Clearly, the relations \eqref{delta_0} and \eqref{delta_1} are analogs of \eqref{reprS} and \eqref{reprSp} for the case of polynomials depending on the spectral parameter in the boundary condition. Thus, the coefficients $\{ G(t), J(t), C_0, \dots, C_p, D_0, \dots, D_{p-1}\}$ play an analogous role in  \eqref{delta_0} and \eqref{delta_1} as the Cauchy data $\{ \mathcal K(t), \mathcal N(t)\}$.

Thus, in this paper, we consider the following problem.

\begin{ip}\label{ip:c}
	Given the generalized Cauchy data $\{ G(t), J(t), C_0, \dots, C_p, D_0, \dots, D_{p-1}\}$, find $\sigma(x)$, $r_1(\la)$, and $r_2(\la)$.
\end{ip}

Note that, using the generalized Cauchy data, one can construct the functions $\Delta_0(\la)$ and $\Delta_1(\la)$ via \eqref{delta_0} and \eqref{delta_1}, respectively, and then find $M(\la)$ by \eqref{fracD01}. Therefore, the uniqueness for solution of Inverse Problem~\ref{ip:c} readily follows from Proposition~\ref{prop:uniq}.

In this paper, we prove the following theorem on local solvability and stability of Inverse Problem~\ref{ip:c}.

\begin{thm} \label{cauchy_thm}
	Let $\tilde\sigma \in L_2(0, \pi)$, $(\tilde r_1, \tilde r_2) \in \mathcal R_p$ and let $\{\tilde G(t), \tilde J(t), \tilde C_0, \dots, \tilde C_{p}, \tilde D_0, \dots, \tilde D_{p-1}\}$ be the generalized Cauchy data for the problem $\tilde L = L(\tilde \sigma, \tilde r_1, \tilde r_2)$. Then, there exists $\varepsilon > 0$ such that, for any $G, J \in L_2(0, \pi)$, $C_0, \dots, C_{p}, D_0, \dots, D_{p-1} \in \mathbb C$ satisfying the condition
	\begin{align} \nonumber
	\delta := \max\{&\|G(t)-\tilde G(t)\|_{L_2(0, \pi)}, \|J(t)-\tilde J(t)\|_{L_2(0, \pi)}, |C_0 - \tilde C_0|, \dots, |C_{p} - \tilde C_{p}|,\\ \label{difCauchy}
	& |D_0 - \tilde D_0|, \dots, |D_{p-1} - \tilde D_{p-1}|\} \le \varepsilon,
	\end{align}
	there exist a complex-valued function $\sigma \in L_2(0, \pi)$ and polynomials $(r_1, r_2)\in\mathcal R_p$, which are the solution of Inverse Problem~\ref{ip:c} by the generalized Cauchy data $\{G(t), J(t), C_0, \dots, C_{p}, D_0, \dots, D_{p-1}\}$. Moreover,
	\begin{equation} \notag
	\|\sigma - \tilde\sigma\|_{L_2(0, \pi)} \le C\delta, \quad |c_n - \tilde c_n| \le C\delta, \quad |d_n - \tilde d_n| \le C\delta, \quad n = \overline{0, p},
	\end{equation}
	where the constant $C$ depends only on $\tilde L$.
\end{thm}

The proof of Theorem~\ref{cauchy_thm} is based on the reduction of the inverse problem by the generalized Cauchy data to the one by the spectral data $\{ \hat M_N(\la) \}_{\la \in \Gamma_N} \cup \{ \la_n, \alpha_n\}_{n \ge N}$ and on the application of Theorem~\ref{stability_thm}.

\section{Preliminaries} \label{sec:prelim}

Throughout the paper, we use the same symbol $C$ for various positive constants independent of $x$, $\la$, $n$, etc. Derivatives with respect to $\la$ are denoted as $f^{\langle j \rangle}(\la) = \dfrac{d^j}{d^j\la} f(\la)$.

Let $\psi(x, \la)$ and $\Phi(x, \la)$ be the solutions of equation \eqref{eqv} satisfying the initial conditions
\begin{equation} \label{initpsi}
\psi(\pi, \la) = r_1(\la), \quad \psi^{[1]}(\pi, \la) = -r_2(\la),
\end{equation}
and the boundary conditions
\begin{gather}\label{bcPhi}
\Phi^{[1]}(0,  \la) = 1, \quad r_1(\la)\Phi^{[1]}(\pi, \la) + r_2(\la)\Phi(\pi, \la)=0,
\end{gather}
respectively. Obviously, for each fixed $x \in [0,\pi]$, the functions $\psi(x, \la)$ and $\psi^{[1]}(x, \la)$ are entire in $\la$. The function $\Phi(x, \la)$ is called the Weyl solution of the problem $L$. One can easily show that 
\begin{gather} \label{Mpsi}
M(\la) = \Phi(0, \la) = -\dfrac{\psi(0,\la)}{\Delta_1(\la)}.
\end{gather}

Consequently, the functions $\Phi(x, \la)$, $\Phi^{[1]}(x, \la)$, and $M(\la)$ are meromorphic in $\la$. All their singularities are poles that coincide with the eigenvalues $\{ \la_n \}_{n \in I}$.

Each eigenvalue $\la_n$ of the boundary value problem $L$ has exactly one corresponding eigenfunction
\begin{equation} \label{yn}
y_n(x) = \psi(x, \la_n) = \beta_n \varphi(x, \la_n),
\end{equation}
where $\beta_n$ is a non-zero constant. If $m_n > 1$, then $\{ \varphi^{\langle j \rangle}(x, \la_n) \}_{j = 0}^{m_n-1}$ and $\{ \psi^{\langle j \rangle}(x, \la_n) \}_{j = 0}^{m_n-1}$ are the chains of associated functions (see \cite{Nai68}). Relying on these facts, we prove the following auxiliary lemma.

\begin{lem} \label{lem:f}
Let $\la_n$ be an eigenvalue of the problem $L$ of multiplicity $m_n$. Then, the function $f(\la) := r_1(\la) - M(\la) \Delta_1(\la) \varphi(\pi, \la)$ has a zero $\la = \la_n$ of multiplicity at least $m_n$.
\end{lem}

\begin{proof}
Let us prove the lemma for $m_n = 2$. In the general case, the arguments are similar but more technically complicated.

Using \eqref{Mpsi}, we get $f(\la) = r_1(\la) + \psi(0,\la) \varphi(\pi, \la)$. Then, the relation $f(\la_n) = 0$ trivially follows from \eqref{yn} and the initial conditions \eqref{initphi}, \eqref{initpsi} for $\varphi(x, \la)$ and $\psi(x, \la)$.

Next, $\varphi^{\langle 1 \rangle}(x, \la_n)$ and $\dfrac{1}{\beta_n}\psi^{\langle 1 \rangle}(x, \la_n)$ are the associated functions corresponding to the same eigenfunction $\varphi(x, \la_n)$. Hence, $\varphi^{\langle 1 \rangle}(x, \la_n) - \dfrac{1}{\beta_n}\psi^{\langle 1 \rangle}(x, \la_n)$ is an eigenfunction, so
$$
\varphi^{\langle 1 \rangle}(x, \la_n) - \dfrac{1}{\beta_n}\psi^{\langle 1 \rangle}(x, \la_n) = c_n\psi(x,\la_n) = c_n\beta_n\varphi(x, \la_n),
$$ 
where $c_n$ is a constant.
Then, using \eqref{initphi}, \eqref{initpsi}, and \eqref{yn}, we obtain
\begin{multline*}
f^{\langle 1 \rangle}(\la_n) = r_1^{\langle 1 \rangle}(\la_n) - \varphi(\pi, \la_n)\psi^{\langle 1 \rangle}(0, \la_n) - \psi(\pi, \la_n)\varphi^{\langle 1 \rangle}(0, \la_n) \\
= r_1^{\langle 1 \rangle}(\la_n) - r_1(\la_n)\big( \varphi^{\langle 1 \rangle}(0, \la_n) - c_n\beta_n\varphi(0, \la_n) \big) - c_n\beta_n\psi(\pi, \la_n) - \psi^{\langle 1 \rangle}(\pi, \la_n) = 0.
\end{multline*}

Thus, $\la_n$ is at least a double zero of $f(\la)$.
\end{proof}

\section{Main equation} \label{sec:maineq}

In this section, we reduce Inverse Problem~\ref{ip:main3} to the linear main equation in the Banach space $C(\Gamma_N)$ of continuous functions, basing on the results of \cite{ChitBond}. Furthermore, we study the properties of the operator $\tilde Q(x)$ in the main equation and of the solution $\varphi(x, \la)$, which are used in the next sections for proving the local solvability of the inverse problem.

Let us consider two boundary value problems $L = L(\sigma, r_1, r_2)$ and $\tilde L = L(\tilde\sigma, \tilde r_1, \tilde r_2)$. Assume that $\sigma, \tilde \sigma \in L_2(0,\pi)$, $(r_1, r_2), (\tilde r_1, \tilde r_2) \in \mathcal R_p$, and $p = \tilde p$. Note that the quasi-derivatives corresponding to $L$ and $\tilde L$ are different: $y^{[1]} = y' - \sigma y$ and $\tilde y^{[1]} = y' - \tilde \sigma y$, respectively.

Introduce the notations
\begin{equation*} 
\def\arraystretch{1.5}
%\left.
\begin{array}{c}
\la_{n0} = \la_n, \quad \la_{n1} = \tilde \la_n, \quad \rho_{n0} = \rho_n, \quad \rho_{n1} = \tilde \rho_n, \quad \alpha_{n0} = \alpha_n, \quad \alpha_{n1} = \tilde \alpha_n, \\
I_0 = I, \quad I_1 = \tilde I, \quad m_{n0} = m_n, \quad m_{n1} = \tilde m_n, \\
\varphi_{n+j, i}(x) = \varphi^{\langle j \rangle}(x, \la_{ni}), \quad \tilde \varphi_{n+j, i}(x) = \tilde \varphi^{\langle j \rangle}(x, \la_{ni}), \quad n \in I_i, \quad j=\overline{0, m_{ni}}.
\end{array} %\quad \right\}
\end{equation*}

In this section, we assume that $\la_{n0} = \la_{n1}$ and $\alpha_{n0} = \alpha_{n1}$ for $n \ge N$ and $\{ \la_n \}_{n = 1}^{N-1}, \, \{ \tilde \la_n \}_{n = 1}^{N-1} \subset \mbox{int} \, \Gamma_N$ for some index $N$. Thus, we consider a finite perturbation of the spectral data.
Then, the results of \cite{ChitBond} imply the following proposition for the solutions $\varphi(x, \la)$ and $\Phi(x, \la)$.

\begin{prop}[\cite{ChitBond}]
	The following representations hold:
	\begin{gather} \label {series_varphi}
	\varphi(x, \la) = \tilde \varphi(x, \la) - \dfrac{1}{2\pi i} \oint_{\Gamma_N} {\tilde D(x, \la, \mu)\varphi(x, \mu) \hat M_N(\mu) d\mu},\\ \label {series_phi}
	\Phi(x, \la) = \tilde \Phi(x, \la) - \dfrac{1}{2\pi i} \oint_{\Gamma_N} {\tilde E(x, \la, \mu)\varphi(x, \mu) \hat M_N(\mu) d\mu},
	\end{gather}
	where
	\begin{equation} \notag
	\tilde D(x, \la, \mu) = \frac { \tilde \varphi(x, \la)\tilde \varphi^{[1]}(x, \mu) - \tilde \varphi^{[1]}(x, \la)\tilde \varphi(x, \mu)}{\la-\mu},
	\end{equation}
	\begin{equation} \notag
	 \tilde E(x, \la, \mu) = \frac { \tilde \Phi(x, \la) \tilde \varphi^{[1]}(x, \mu) - \tilde \Phi^{[1]}(x, \la) \tilde \varphi(x, \mu)}{\la-\mu}.
	\end{equation}
\end{prop}

Recall that $C(\Gamma_N)$ is the space of the continuous functions on the contour $\Gamma_N$ with the norm $\|f\|_{C(\Gamma_N)} = \max\limits_{\la \in \Gamma_N} |f(\la)|$. For each fixed $x \in  [0,\pi]$, define a linear operator $\tilde Q(x) \colon C(\Gamma_N) \to C(\Gamma_N)$ as follows: 
\begin{gather}\label{operator_Q}
(\tilde Q(x) f)(\la) = \dfrac{1}{2\pi i}\oint\limits_{\Gamma_N}{\tilde D(x, \la, \mu)\hat M_N(\mu)f(\mu) \, d\mu}, \quad f \in C(\Gamma_N).
\end{gather}

Then, the relation \eqref{series_varphi} can be rewritten as
\begin{gather}\label{main_varphi}
\tilde\varphi(x, \la) = (E+\tilde Q(x))\varphi(x, \la),
\end{gather}
where $E$ is the identity operator in $C(\Gamma_N)$. Equation \eqref{main_varphi} is called the main equation of Inverse Problem~\ref{ip:main3}. Indeed, if we know the problem $\tilde L$ and the spectral data $\{ \la_n, \alpha_n\}_{n \ge 1}$ of the problem $L$ such that $\la_n = \tilde \la_n$, $\alpha_n = \tilde \alpha_n$ for $n \ge N$, then we can construct the function $\tilde \varphi(x, \la)$ and the operator $\tilde Q(x)$, find $\varphi(x, \la)$ by solving equation \eqref{main_varphi}, and use $\varphi(x, \la)$ to obtain the coefficients $\sigma$, $r_1$, and $r_2$ of the problem $L$. In order to prove the unique solvability of the main equation \eqref{main_varphi}, we need the following lemma.

\begin{lem} \label{Q_est}
$\tilde Q(x)$ is continuous with respect to $x \in [0,\pi]$ in the space of linear bounded operators from $C(\Gamma_N)$ to $C(\Gamma_N)$. Moreover,
the following estimate holds:
\begin{equation} \label{estQ}
\|\tilde Q(x)\|_{C(\Gamma_N) \to C(\Gamma_N)} \le C\delta_1, \quad x \in [0,\pi],
\end{equation}
where $\delta_1 = \max\limits_{\la \in \Gamma_N} |\hat M_N(\la)|$ and $C$ depends only on the problem $\tilde L$.
\end{lem}

\begin{proof}
By definition, we have
\begin{equation} \label{defnorm}
\|\tilde Q(x)\|_{C(\Gamma_N) \to C(\Gamma_N)} = \sup\limits_{f \in C(\Gamma_N), f \neq 0} \dfrac{\|\tilde Q(x) f\|_{C(\Gamma_N)}}{\|f\|_{C(\Gamma_N)}}.
\end{equation}

The function $\tilde D(x, \la, \mu)$ can be represented in the following integral form:
$$
\tilde D(x, \la, \mu) = \int\limits_{0}^{x} \tilde\varphi(t, \la)\tilde\varphi(t, \mu)dt,
$$
so this function is entire in $\la$ and $\mu$. 
Moreover, $\tilde\varphi(x, \la)$ and $\tilde D(x, \la, \mu)$ are continuous over the sets of their variables. This together with \eqref{operator_Q} and \eqref{defnorm} imply the continuity of $\tilde Q(x)$ by $x \in [0,\pi]$.
Furthermore, we obtain
$$
\|\tilde Q(x) f\|_{C(\Gamma_N)} \le C\max\limits_{\la, \mu \in \Gamma_N}|\tilde D(x, \la, \mu)|\max\limits_{\mu \in \Gamma_N}|\hat M_N(\la)|\max\limits_{\la \in \Gamma_N} |f(\la)| \le C\delta_1\max\limits_{\la \in \Gamma_N} |f(\la)|.
$$

Combining the latter estimate with \eqref{defnorm}, we arrive at \eqref{estQ}
\end{proof}

\begin{cor} \label{unique_solv}
If $\delta_1 > 0$ is sufficiently small, then the operator $(E + \tilde Q(x))$ has a bounded inverse operator in $C(\Gamma_N)$ and so the main equation \eqref{main_varphi} is uniquely solvable for each fixed $x \in [0,\pi]$.
\end{cor}

\begin{lem} \label{lem:varphi}
Under the conditions of Corollary~\ref{unique_solv},	
the solution $\varphi(x, \la)$ of the main equation \eqref{main_varphi} is continuous with respect to $x \in [0,\pi]$ in the space $C(\Gamma_N)$.
Furthermore, $|\varphi(x, \la)| \le C$ for $x \in [0,\pi]$ and $\la \in \Gamma_N$, where $C$ does not depend on $x$ and $\la$.
\end{lem}

\begin{proof}
Put $\tilde P(x) = (E+\tilde Q(x))^{-1}$. Due to Corollary~\eqref{unique_solv}, we get that $\tilde P(x)$ is bounded for each $x \in [0, \pi]$. 
Fix $x_0 \in [0,\pi]$. By Lemma~\ref{Q_est},
the operator $\tilde Q(x)$ is continuous at the point $x = x_0$, so there exists $\varepsilon > 0$ such that, for each $x \in [0, \pi]$ satisfying $|x-x_0|<\varepsilon$, the following estimate holds: 
$$
\|\tilde Q(x_0) - \tilde Q(x)\| \le \dfrac{1}{2 \|\tilde P(x_0)\|}. 
$$
Then
$$
\tilde P(x) - \tilde P(x_0) = \sum\limits_k(\tilde Q(x_0) - \tilde Q(x))^k(\tilde P(x_0))^{k+1}.
$$
Hence
$$
\|\tilde P(x) - \tilde P(x_0)\| \le 2\|\tilde P(x_0)\|^2\|\tilde Q(x_0) - \tilde Q(x)\| \to 0,
$$
as $x \to x_0$. Since the point $x_0$ can be arbitrary, we conclude that 
$\tilde P(x)$ is continuous with respect to $x \in [0,\pi]$ in the space of linear operators from $C(\Gamma_N)$ to $C(\Gamma_N)$ 
and $\|\tilde P(x)\| \le C$. Note that the assertion of this lemma holds for the solution $\tilde \varphi(x, \la)$ corresponding to the problem $\tilde L$. Using the relation $\varphi(x, \la) = \tilde P(x) \tilde \varphi(x, \la)$, we arrive at the same assertion for $\varphi(x, \la)$.
\end{proof}

\section{Reconstruction formulas} \label{sec:rec}

In this section, we prove Theorem~\ref{mainthm} on the reconstruction formulas, which can be used for the recovery of $\sigma(x)$, $r_1(\la)$, and $r_2(\la)$ from the solution $\varphi(x, \la)$ of the main equation \eqref{main_varphi}. The formula \eqref{sigma_series} for $\sigma(x)$ is obtained analogously to the previous studies \cite{BondTamkang, ChitBond}, so here we mostly focus on construction of the polynomials. The reconstruction formulas \eqref{r1_series} and \eqref{r2_series} are deduced from the relation \eqref{series_varphi} for $\Phi(x, \la)$.

As in the previous section, we suppose that $\la_{n0} = \la_{n1}$, $\alpha_{n0} = \alpha_{n1}$ for $n \ge N$ and $\la_{nj} \in \mbox{int} \, \Gamma_N$ for $n < N$, $j = 0, 1$. Let us represent the right boundary condition \eqref{bcPhi} for $\tilde\Phi(x, \la)$ and $\Phi(x, \la)$ in the following form, supposing that $\tilde\Phi(\pi, \la) \neq 0$ and $\Phi(\pi, \la) \neq 0$:
\begin{equation} \label{fracPhi}
\dfrac{\tilde r_2(\la)}{\tilde r_1(\la)} = - \dfrac{\tilde\Phi^{[1]}(\pi, \la)}{\tilde\Phi(\pi, \la)}, \quad \dfrac{r_2(\la)}{r_1(\la)} = - \dfrac{\Phi^{[1]}(\pi, \la)}{\Phi(\pi, \la)}.
\end{equation}

Taking the quasi-derivative of $\Phi(x, \la)$ and due the \eqref{series_phi}, we can get that
\begin{align} \nonumber
\Phi^{[1]}(\pi, \la) = & \tilde\Phi^{[1]}(\pi, \la) + \tilde\Phi(\pi, \la)\dfrac{1}{2\pi i}\oint\limits_{\Gamma_N} {(\tilde\varphi(\pi, \mu)\varphi(\pi, \mu) - 1)\hat M_N(\mu)}d\mu - \\ \label{quasiPhi}
&\dfrac{1}{2\pi i}\oint\limits_{\Gamma_N} {\tilde E(\pi, \la, \mu)\varphi^{[1]}(\pi, \la)\hat M_N(\mu)}d\mu.
\end{align}

Using \eqref{series_phi}, \eqref{fracPhi}, and \eqref{quasiPhi}, we obtain
\begin{equation} \label{fracr12}
\dfrac{r_2(\la)}{r_1(\la)} = \dfrac{Y(\la)}{Z(\la)},
\end{equation}
where
\begin{align} \label{Z_series}
	Z(\la) := \tilde r_1(\la) - \dfrac{1}{2\pi i}\oint_{\Gamma_N}{\dfrac{\tilde r_1(\la)\tilde\varphi^{[1]}(\pi, \mu) + \tilde r_2(\la)\tilde\varphi(\pi, \mu)}{\la - \mu}\varphi(\pi, \mu)\hat M_N(\mu)}d\mu,
	\end{align}
	\begin{align}\notag 
	Y(\la) & := \tilde r_2(\la) - \tilde r_1(\la)\frac{1}{2\pi i}\oint_{\Gamma_N}{\big(\tilde\varphi(\pi, \mu)\varphi(\pi, \mu) - 1\big)\hat M_N(\mu)}d\mu\\ \notag
	&+\dfrac{1}{2\pi i}\oint_{\Gamma_N}{\dfrac{\tilde r_1(\la)\tilde\varphi^{[1]}(\pi, \mu) + \tilde r_2(\la)\tilde\varphi(\pi, \mu)}{\la - \mu}\varphi^{[1]}(\pi, \mu)\hat M_N(\mu)}d\mu.
	\end{align}

However, the functions $Z(\la)$ and $Y(\la)$ in general are not polynomials. For obtaining polynomials of the desired degrees ($\deg r_1 = p$, $\deg r_2 \le p$), we have to multiply $Z(\la)$ and $Y(\la)$ by the product $\prod\limits_{k = 1}^{N-1} \dfrac{\la - \la_k}{\la - \tilde \la_k}$. In order to justify this multiplication, we need the following lemma. 

\begin{lem}\label{lem:cz}
$Y^{\langle k \rangle}(\la_{n1})=Z^{\langle k \rangle}(\la_{n1})=0$, $n \in I$, $n < N$, $k = \overline{0, m_{n1}-1}$.
\end{lem}

\begin{proof}
Let us prove the lemma for $Z(\la)$. The proof for $Y(\la)$ is analogous.
From \eqref{series_varphi} and \eqref{Z_series}, we get
\begin{gather} \label {varphi_res}
\varphi(\pi, \la) = \tilde \varphi(\pi, \la) - \sum\limits_{\la_{kj} \in int \, \Gamma_N} \mathop{\mathrm{Res}}_{\mu=\la_{kj}}\dfrac{\tilde\varphi(\pi, \la)\tilde\varphi^{[1]}(\pi, \mu)-\tilde\varphi^{[1]}(\pi, \la)\tilde\varphi(\pi, \mu)}{\la - \mu}\varphi(\pi, \mu)\hat M_N(\mu) \\ \label{Z_res}
Z(\la) = \tilde r_1(\la) - \sum\limits_{\la_{kj} \in int \, \Gamma_N} \mathop{\mathrm{Res}}_{\mu=\la_{kj}}{\dfrac{\tilde r_1(\la)\tilde\varphi^{[1]}(\pi, \mu) + \tilde r_2(\la)\tilde\varphi(\pi, \mu)}{\la - \mu}\varphi(\pi, \mu)\hat M_N(\mu)}.
\end{gather}

Multiplying \eqref{varphi_res} by $\dfrac{\tilde r_1(\la)}{\tilde\varphi(\pi, \la)}$ and using \eqref{charfun}, we obtain
\begin{align*}
\dfrac{\tilde r_1(\la)\varphi(\pi, \la)}{\tilde\varphi(\pi, \la)} &= \tilde r_1(\la) - \sum\limits_{\la_{kj} \in int \, \Gamma_N} \mathop{\mathrm{Res}}_{\mu=\la_{kj}}{\dfrac{\tilde r_1(\la)\tilde\varphi^{[1]}(\pi, \mu) + \tilde r_2(\la)\tilde\varphi(\pi, \mu)}{\la - \mu}\varphi(\pi, \mu)\hat M_N(\mu)} + \\
&\dfrac{\tilde\Delta(\la)}{\tilde\varphi(\pi, \la)} \sum\limits_{\la_{kj} \in int \, \Gamma_N} \mathop{\mathrm{Res}}_{\mu=\la_{kj}}\dfrac{\tilde\varphi(\pi, \mu)}{\la - \mu}\varphi(\pi, \mu)\hat M_N(\mu).
\end{align*}
Consequently,
\begin{gather} \label{z_la}
\dfrac{\tilde r_1(\la)\varphi(\pi, \la)}{\tilde\varphi(\pi, \la)} = Z(\la) + \dfrac{\tilde\Delta_1(\la)}{\tilde\varphi(\pi, \la)} \sum\limits_{\la_{kj} \in int \, \Gamma_N} \mathop{\mathrm{Res}}_{\mu=\la_{kj}}\dfrac{\tilde\varphi(\pi, \mu)}{\la - \mu}\varphi(\pi, \mu)\hat M_N(\mu)
\end{gather}

Counting residues in a general case, due the \eqref{mn_func} we can get
$$
\mathop{\mathrm{Res}}_{\mu=\la_{kj}}\dfrac{\tilde\varphi(\pi, \mu)}{\la - \mu}\varphi(\pi, \mu)\hat M_N(\mu) = \sum\limits_{i=0}^{m_{kj}-1}\alpha_{k+i, j}\sum\limits_{u=0}^{i}(\la-\la_{kj})^{u-i-1}\sum\limits_{s=0}^{u}\dfrac{\tilde\varphi_{k+s, j}(\pi)\varphi_{k+u-s, j}(\pi)}{s!(u-s)!}.
$$

Substituting $\la = \la_{n1}$ into \eqref{z_la}, we see that
$$
\dfrac{\tilde\Delta_1(\la_{n1})}{\tilde\varphi(\pi, \la_{n1})} \sum\limits_{\la_{kj} \in int \Gamma_N} \mathop{\mathrm{Res}}_{\mu=\la_{kj}}\dfrac{\tilde\varphi(\pi, \mu)}{\la - \mu}\varphi(\pi, \mu)\hat M_N(\mu) = 0 \quad \text{if}\:\la_{kj}\neq \la_{n1}.
$$

So, we need to find the residues only at $\la_{n1}$. Using \eqref{double_sum} and the powers of $(\la - \la_{n1})$ in the numerator and the denominators into account, we derive
\begin{align*}
\lim_{\la \to \la_{n1}} \dfrac{\tilde\Delta_1(\la)}{\tilde\varphi(\pi, \la)} \sum\limits_{i=0}^{m_{n1}-1}\alpha_{n+i, 1}\sum\limits_{u=0}^{i}(\la-\la_{n1})^{u-i-1}\sum\limits_{s=0}^{u}\dfrac{\tilde\varphi_{n+s, 1}(\pi)\varphi_{n+u-s, 1}(\pi)}{s!(u-s)!} = \\
\alpha_{n+m_{n1}-1, 1}\lim_{\la \to \la_{n1}} {\dfrac{\tilde\Delta_1(\la)\varphi_{n1}(\pi)}{(\la - \la_{n1})^{m_{n1}}}}.
\end{align*}

Using  L’H\^opital’s Rule, we get
$$
\lim_{\la \to \la_{n1}} {\dfrac{\tilde\Delta_1(\la)\varphi_{n1}(\pi)}{(\la - \la_{n1})^{m_{n1}}}} = \dfrac{\tilde\Delta_1^{\langle m_{n1} \rangle}(\la_{n1})\varphi_{n1}(\pi)}{m_{n1}!},
$$
and, consequently, substituting this into \eqref{z_la}, we obtain
\begin{gather} \label{for_z_zero}
\dfrac{\tilde r_1(\la_{n1})\varphi_{n1}(\pi)}{\tilde\varphi_{n1}(\pi)} = Z(\la_{n1}) - \dfrac{\tilde\Delta_1^{\langle m_{n1} \rangle}(\la_{n1})\varphi_{n1}(\pi)}{m_{n1}!}\alpha_{n+m_{n1}-1, 1}.
\end{gather}

Next, the relations \eqref{double_sum} and \eqref{Mpsi} imply
\begin{multline} \label{calc_alpha}
\alpha_{n+m_{n1}-1, 1} = \mathop{\mathrm{Res}}_{\la=\la_{n1}} (\la-\la_{n1})^{m_{n1}-1}\tilde M(\la) \\ =-\mathop{\mathrm{Res}}_{\la=\la_{n1}} \dfrac{\tilde\psi(0, \la)}{\dfrac{\tilde\Delta_1(\la)}{(\la-\la_{n1})^{m_{n1}-1}}} =- \dfrac{\tilde\psi_{n1}(0)}{\Bigg(\dfrac{\tilde\Delta_1(\la)}{(\la-\la_{n1})^{m_{n1}-1}}\Bigg)^{\langle 1 \rangle}\Bigg|_{\la = \la_{n1}}},
\end{multline}
where $\tilde\psi_{n1}(x) = \tilde \psi(x, \la_{n1})$.
Using  L’H\^opital’s Rule, we get 
$$
\lim_{\la \to \la_{n1}} {\Bigg(\dfrac{\tilde\Delta_1(\la)}{(\la-\la_{n1})^{m_{n1}-1}}\Bigg)^{\langle 1 \rangle}} = \dfrac{\tilde\Delta_1^{\langle m_{n1} \rangle}(\la_{n1})}{m_{n1}!}.
$$

Using \eqref{yn}, we have  $\tilde\psi_{n1}(x) = \tilde\beta_n\tilde\varphi_{n1}(x)$, where $\tilde\beta_n \ne 0$. Hence
\begin{equation} \label{psin1}
\tilde\psi_{n1}(0) = \tilde \beta_n\tilde\varphi_{n1}(0) = \dfrac{\tilde\psi_{n1}(\pi)}{\tilde\varphi_{n1}(\pi)} = \dfrac{\tilde r_1(\la_{n1})}{\tilde\varphi_{n1}(\pi)}.
\end{equation}
Combining \eqref{calc_alpha} and \eqref{psin1}, we get
\begin{gather}\notag
\alpha_{n+m_{n1}-1, 1} = -\dfrac{m_{n1}\tilde r_1(\la_{n1})}{\tilde\Delta_1^{\langle m_{n1} \rangle}(\la_{n1})\tilde\varphi_{n1}(\pi)}.
\end{gather}
Substituting the latter equality into \eqref{for_z_zero}, we get that $Z(\la_{n1}) = 0$.

For derivatives, this way is technically hard, so we use an alternative method.
Counting residues in \eqref{z_la}, we obtain
\begin{multline*}
\tilde\varphi(\pi, \la) Z(\la) = \tilde r_1(\la)\varphi(\pi, \la) \\ + \tilde\Delta_1(\la)\sum\limits_{\la_{kj} \in int \, \Gamma_N}\sum\limits_{i=0}^{m_{kj}-1}\alpha_{k+i, j}\sum\limits_{u=0}^{i}(\la-\la_{kj})^{u-i-1}\sum\limits_{s=0}^{u}\dfrac{\tilde\varphi_{k+s, j}(\pi)\varphi_{k+u-s, j}(\pi)}{s!(u-s)!}
\end{multline*}

Passing to the Taylor series, we get for $u = \overline{1,m_{n1}-1}$ that
\begin{align*}
\sum\limits_{k=0}^{u-1}\dfrac{\tilde\varphi_{n+k, 1}(\pi)Z^{\langle u-k \rangle}(\la_{n1})}{k!(u-k)!} &= \sum\limits_{k=0}^{u}\dfrac{\tilde r_1^{\langle k \rangle}(\la_{n1})\varphi_{n+u-k1}(\pi)}{k!(u-k)!} +\\
& \sum\limits_{k=0}^{u}\dfrac{\tilde\Delta_1^{\langle m_{n1}+u-k \rangle}(\la_{n1})}{(m_{n1}+u-k)!}\sum\limits_{i=0}^{k}\alpha_{n+m_{n1}-1-i, 1}\sum\limits_{s=0}^{k-i}\dfrac{\tilde\varphi_{n+s, 1}(\pi)\varphi_{n+i-s, 1}(\pi)}{s!(i-s)!}.
\end{align*}

Introduce the function $\tilde f(\la) = \tilde r_1(\la) - \tilde M(\la)\tilde\Delta_1(\la)\tilde\varphi(\pi, \la)$. From the Taylor series expansion, we obtain
\begin{gather*}
\tilde f^{\langle k \rangle} (\la_{n1}) = \tilde r^{\langle k \rangle}_1(\la_{n1}) + \sum\limits_{i=0}^{k}\alpha_{n+m_{n1}-1-i, 1}\sum\limits_{j=0}^{k-i}\dfrac{\tilde\Delta_1^{\langle m_{n1}+j \rangle}(\la_{n1})\tilde\varphi_{n+k-i-j, 1}(\pi)}{(m_{n1}+j)!(k-i-j)!}, \quad k=\overline{0, m_{n1} - 1}.
\end{gather*}

Comparing this expansion with the one for $\tilde\varphi(\pi, \la)Z(\la)$, we get
\begin{gather*}
\varphi_{n1}(\pi)\tilde f(\la_{n1}) = 0, \\
\varphi^{\langle 1 \rangle}_{n1}(\pi)\tilde f(\la_{n1}) + \varphi_{n1}(\pi)\tilde f^{\langle 1 \rangle}(\la_{n1}) = \tilde\varphi_{n1}(\pi)Z^{\langle 1 \rangle}(\la_{n1}), \\
\dots,\\
\big(\varphi(\pi, \la)\tilde f(\la)\big)^{\langle m_{n1}-1 \rangle}\bigg|_{\la = \la_{n1}} = \sum\limits_{k=0}^{m_{n1}-2}\dfrac{\tilde\varphi_{n+k, 1}(\pi)Z^{\langle m_{n1}-1-k \rangle}(\la_{n1})}{k!(m_{n1}-1-k)!}
\end{gather*}

By virtue of Lemma~\ref{lem:f}, $\tilde f^{\langle k \rangle}(\la_{n1}) = 0$, $k=\overline{0, m_{n1}-1}$. Consequently, 
$Z^{\langle k \rangle}(\la_{n1}) = 0$, $k=\overline{0, m_{n1}-1}$, which concludes the proof.
\end{proof}

Put
\begin{equation} \label{defg12}
g_1(\la) := Z(\la) \prod\limits_{k = 1}^{N-1} \dfrac{\la-\la_{k0}}{\la-\la_{k1}}, \quad g_2(\la) := Y(\la) \prod\limits_{k = 1}^{N-1} \dfrac{\la-\la_{k0}}{\la-\la_{k1}}.
\end{equation}

According to Lemma~\ref{lem:cz}, $g_1(\la)$ and $g_2(\la)$ are entire functions.

\begin{lem} \label{degree}
$(g_1, g_2) \in \mathcal R_p$, where $g_1(\la)$ and $g_2(\la)$ are defined by \eqref{defg12}.
\end{lem}

\begin{proof}
One can easily check that each of the residues in the formula \eqref{Z_res} can be represented in the form
$$
\dfrac{\tilde r_1(\la) A_{m_{kj} - 1}(\la) + \tilde r_2(\la) B_{m_{kj} - 1}(\la)}{(\la - \la_{kj})^{m_{kj}}},
$$
where $A_{m_{kj} - 1}(\la)$ and $B_{m_{kj} - 1}(\la)$ are polynomials of degree $m_{kj} - 1$. Then, the sum of these residues has a polynomial of degree not greater than $2(N-1)-2+p = 2N+p-4$ in the numerator.
So, we get
\begin{align*}
g_1(\la) = & \prod\limits_{k = 1}^{N-1} \dfrac{\la-\la_{k0}}{\la-\la_{k1}} \left( \tilde r_1(\la) - \sum\limits_{\la_{kj} \in int \, \Gamma_N} \mathop{\mathrm{Res}}_{\mu=\la_{kj}}{\dfrac{\tilde r_1(\la)\tilde\varphi^{[1]}(\pi, \mu) + \tilde r_2(\la)\tilde\varphi(\pi, \mu)}{\la - \mu}\varphi(\pi, \mu)\hat M_N(\mu)}\right) \\ =
& \dfrac{\sum\limits_{n=0}^{2N+p-3}s_n\la^n + \la^{2N+p-2}}{\prod\limits_{k = 1}^{N-1} (\la-\la_{k1})^{2}},
\end{align*}
where $s_n \in \mathbb C$, $n = \overline{0, 2N+p-3}$.

By virtue of Lemma~\ref{lem:cz}, this fraction can be reduced, so we get a polynomial of degree
$$
\mbox{deg} \, (g_1(\la)) = 2N+p-2 - 2(N-1) = p,
$$
and its leading coefficient equals $1$.
In the same way, one can prove that $g_2(\la)$ is a polynomial of degree not greater than $p$.
\end{proof}

The relations \eqref{fracr12}, \eqref{defg12} together with Lemma~\ref{degree} imply that $r_1(\la) \equiv g_1(\la)$, $r_2(\la) \equiv g_2(\la)$. This proves
the reconstruction formulae \eqref{r1_series} and \eqref{r2_series}.

\section{Local solvability and stability} \label{sec:solvstab}

In this section, we provide the proof of Theorem~\ref{stability_thm} on local solvability and stability of Inverse Problem~\ref{ip:main3}. The proof consists of the two steps:
\begin{enumerate}
	\item Perturbation of finite spectral data for $n < N$.
	\item Perturbation of simple eigenvalues for $n \ge N$.
\end{enumerate}

At the first step, we formulate Theorem~\ref{thm_mult} and prove it by relying on the results of Sections~\ref{sec:maineq} and~\ref{sec:rec}. Namely, we construct the solution $\varphi(x, \la)$ of the main equation \eqref{main_varphi}, find $\sigma(x)$, $r_1(\la)$, and $r_2(\la)$ by the reconstruction formulas of Theorem~\ref{mainthm}, and obtain all the needed properties for these functions. At the second step, we formulate the result for perturbation of simple eigenvalues from the previous study \cite{ChitBondStab} (see Theorem~\ref{stability_thm_m}). Finally, combining the two steps, we arrive at Theorem~\ref{stability_thm}.

\begin{thm} \label{thm_mult}
Let $\tilde L = L(\tilde \sigma, \tilde r_1, \tilde r_2)$ be a fixed boundary value problem of form \eqref{eqv}-\eqref{bc} with $\tilde \sigma \in L_2(0,\pi)$ and $(\tilde r_1, \tilde r_2) \in \mathcal R_p$, and let $N$ be the index defined in Notation~\ref{not:GammaN}.
Then, there exists $\delta_0 > 0$, depending on the problem $\tilde L$, such that, for any complex numbers $\{\la_n, \alpha_n\}_{n \ge 1}$ satisfying the conditions
\begin{gather*}
	\delta_1 := \max\limits_{\la \in \Gamma_N}|\hat M_N(\la)| \le \delta_0, \\
	\la_n = \tilde \la_n, \quad \alpha_n = \tilde \alpha_n, \quad n \ge N, 
\end{gather*}
there exist a complex-valued function $\sigma(x) \in L_2(0, \pi)$ and polynomials $(r_1, r_2)\in\mathcal R_p$, which are the solution of Inverse Problem~\ref{ip:main3} for the data $\{\la_n, \alpha_n\}_{n \ge 1}$. Moreover, the difference $(\sigma - \tilde \sigma)(x)$ is continuous on $[0,\pi]$ and
\begin{equation} \notag
	\|\sigma - \tilde\sigma\|_{C{[0, \pi]}} \le C\delta_1, \quad |c_n - \tilde c_n| \le C\delta_1, \quad |d_n - \tilde d_n| \le C\delta_1, \quad n = \overline{0, p},
\end{equation}
where the constant $C$ depends only on $\tilde L$.
\end{thm}

The proof of Theorem~\ref{thm_mult} relies on several technical lemmas. Suppose that the problem $\tilde L$ and the data $\{ \la_n, \alpha_n\}_{n \ge 1}$ fulfill the hypothesis of Theorem~\ref{thm_mult} for so small $\delta_0 > 0$ that any $\delta_1 \le \delta_0$ satisfies the conditions of Corollary~\ref{unique_solv}. By using $\tilde L$ and $\{ \la_n, \alpha_n \}_{n = 1}^{N-1}$, construct the main equation \eqref{main_varphi}. By virtue of Corollary~\ref{unique_solv}, equation \eqref{main_varphi} is uniquely solvable. Furthermore, its solution $\varphi(x, \la)$ satisfies the conclusion of Lemma~\ref{lem:varphi}. Next, construct the function $\sigma(x)$ via \eqref{sigma_series}. The properties of this function are described by the following lemma.

\begin{lem} \label{est_sigma}
$\sigma \in L_2(0, \pi)$, 
$\hat\sigma(x) := \sigma(x) - \tilde\sigma(x) \in C[0, \pi]$ and $\|\hat\sigma(x)\|_{C{[0, \pi]}} \le C\delta_1$, where $C$ depends only on $\tilde L$.
\end{lem}

\begin{proof}
From \eqref{sigma_series} we get:
\begin{equation} \label{hatsigma}
\hat\sigma(x) = \dfrac{1}{2\pi i} \oint\limits_{\Gamma_N} {(1-2\tilde\varphi(x, \mu)\varphi(x, \mu))\hat M_N(\mu)}d\mu.
\end{equation}
The assertion of this lemma readily follows from \eqref{hatsigma} and Lemma~\ref{lem:varphi}.
\end{proof}

In addition to $\varphi(x, \la)$, for further constructions, we need the quasi-derivative $\varphi^{[1]}(x, \la) = \varphi'(x, \la) - \sigma(x) \varphi(x, \la)$. However, the existence of the derivative $\varphi'(x, \la)$ does not follow from the above arguments and has to be justified.

\begin{lem} \label{quasider}
The functions $\varphi(x, \la)$ and $\varphi^{[1]}(x, \la)$ are absolutely continuous with respect to $x \in [0,\pi]$ for each fixed $\la \in \Gamma_N$. Moreover, $\varphi^{[1]}(\pi, \la)$ belongs to $C(\Gamma_N)$ and
$|\varphi^{[1]}(\pi, \la)| \le C$ for $\la \in \Gamma_N$, where $C$ depends only on $\tilde L$.
\end{lem}

\begin{proof}
From \eqref{main_varphi} we can get by formal calculations that
\begin{gather} \label{main_quasi_varphi}
\tilde\varphi^{[1]}(x, \la) = (E+\tilde Q(x))\varphi^{[1]}(x, \la) + (E+\tilde Q(x))\hat\sigma(x)\varphi(x, \la)+\tilde Q'(x)\varphi(x, \la).
\end{gather}
Hence
\begin{equation} \label{defvarphi1}
\varphi^{[1]}(x, \la) = \tilde P(x)\tilde\varphi^{[1]}(x, \la)-\hat\sigma(x)\varphi(x, \la)-\tilde P(x)\tilde Q'(x)\varphi(x, \la).
\end{equation}

From \eqref{operator_Q} we can get
$$
\tilde Q'(x)\varphi(x, \la) = \tilde\varphi(x, \la)\dfrac{1}{2\pi i}\oint\limits_{\Gamma_N} {\tilde\varphi(x, \mu)\hat M_N(\mu)\varphi(x, \mu)}d\mu.
$$
Consequently, $\tilde Q'(x)\varphi(x, \la)$ is continuous with respect to $x \in [0,\pi]$ in the space $C(\Gamma_N)$. Note that $\tilde P(x)$, $\tilde \varphi^{[1]}(x, \la)$, and $\hat \sigma(x)$ are also continuous with respect to $x \in [0,\pi]$ in the space of linear operators from $C(\Gamma_N)$ to $C(\Gamma_N)$, in $C(\Gamma_N)$, and as a complex-valued function, respectively. Hence, the right-hand side of \eqref{defvarphi1} is continuous by $x \in [0,\pi]$ in $C(\Gamma_N)$. Therefore, the relation \eqref{defvarphi1} can be used as the definition of $\varphi^{[1]}(x, \la)$. Hence, the function $\varphi'(x, \la) = \varphi^{[1]}(x, \la) + \sigma(x) \varphi(x, \la)$ is integrable by $x$ on $(0,\pi)$ for each fixed $\la \in \Gamma_N$.

From \eqref{main_quasi_varphi} we obtain by formal calculations that
\begin{multline} \label{varphi1dif}
(\varphi^{[1]}(x, \la))' = \tilde P(x)(\tilde\varphi^{[1]}(x, \la))' - \tilde P(x)\tilde Q'(x)(2\varphi^{[1]}(x, \la) + \sigma(x)\varphi(x, \la) + \hat\sigma(x)\varphi(x, \la)) \\ 
-\hat\sigma'(x)\varphi(x, \la)-\hat\sigma(x)\varphi'(x, \la) - \tilde P(x)\tilde Q''(x)\varphi(x, \la).
\end{multline}

Let us consider the behavior of the terms in the right-hand side of \eqref{varphi1dif} with respect to $x$.
As $\tilde\varphi^{[1]}(x, \la) \in AC{[0, \pi]}$, then $(\tilde\varphi^{[1]}(x, \la))' \in L_1(0, \pi)$ as a function of $x$.
Next,
\begin{multline*}
\tilde Q''(x)\varphi(x, \la) = \dfrac{1}{2\pi i}\oint_{\Gamma_N} (2\tilde\sigma(x)\tilde\varphi(x, \la)\tilde\varphi(x, \mu) + \tilde\varphi^{[1]}(x, \la)\tilde\varphi(x, \mu) \\ + \tilde\varphi(x, \la)\tilde\varphi^{[1]}(x, \mu)) \hat M_N(\mu) \varphi(x, \mu) d\mu \in L_2(0,\pi).
\end{multline*}
\vspace*{-10pt}
\begin{multline*}
\hat\sigma'(x) = -\dfrac{1}{\pi i}\oint_{\Gamma_N} {(\tilde\varphi^{[1]}(x, \mu)\varphi(x, \mu) + \tilde\varphi(x, \mu)\varphi^{[1]}(x, \mu))\hat M_N(\mu)}d\mu \\
-(\tilde\sigma(x) + \sigma(x))\dfrac{1}{\pi i}\oint_{\Gamma_N} {\tilde\varphi(x, \mu)\varphi(x, \mu)\hat M_N(\mu)}d\mu,
\end{multline*}
so $\hat\sigma' \in L_2(0, \pi)$ and $\hat\sigma \in AC{[0, \pi]}$.
Consequently, $(\varphi^{[1]}(x, \la))' \in L_1(0, \pi)$ and $\varphi^{[1]}(x, \la) \in AC[0, \pi]$ for each fixed $\la \in C(\Gamma_N)$.

The estimate for $|\varphi^{[1]}(\pi, \la)|$ is obtained from the relation \eqref{defvarphi1} and Lemmas~\ref{Q_est}, \ref{lem:varphi}, \ref{est_sigma}.
\end{proof}

Next, using the continuous functions $\varphi(\pi, \la)$ and $\varphi^{[1]}(\pi, \la)$ for $\la \in \Gamma_N$, we find $r_1(\la)$ and $r_2(\la)$ by the formulae \eqref{r1_series} and \eqref{r2_series}, respectively. Using \eqref{series_varphi}, we can analytically extend the functions $\varphi(x, \la)$ and $\varphi^{[1]}(x, \la)$ into the whole $\la$-plane. In particular, $\varphi(\pi, \la)$ and $\varphi^{[1]}(\pi, \la)$ become analytic inside the contour $\Gamma_N$, so the integrals in \eqref{r1_series} and \eqref{r2_series} can be computed by the Residue Theorem. Consequently, the proofs of Lemmas~\ref{lem:cz} and \ref{degree} remain valid in this case. By Lemma~\ref{degree}, the constructed functions $r_1(\la)$ and $r_2(\la)$ are polynomials and $(r_1, r_2) \in \mathcal R_p$. The stability of reconstruction for the coefficients $\{ c_n \}$ and $\{ d_n \}$ of these polynomials is established by the following lemma.

\begin{lem} \label{polynomials_est}
$|c_n - \tilde c_n| \le C\delta_1$, $|d_n - \tilde d_n| \le C\delta_1$, $n = \overline{0, p}$.
\end{lem}

\begin{proof}

At first, let us estimate the term
$$
\Bigg|1 - \prod\limits_{k = 1}^{N-1} \dfrac{\la-\la_{k0}}{\la-\la_{k1}}\Bigg|.
$$

Denote
\begin{align*}
Q(\la) & = \prod\limits_{k = 1}^{N-1} (\la-\la_{k0}) = \sum\limits_{k=0}^{N-1}q_k\la^k, \\
\tilde Q(\la) & = \prod\limits_{k = 1}^{N-1} (\la-\la_{k1}) = \sum\limits_{k=0}^{N-1}\tilde q_k\la^k.
\end{align*}

The functions $\tilde M_N(\la)$ and $M_N(\la)$ (for sufficiently small $\delta_1$) satisfy the conditions of Lemma A3 from \cite{BondGaidel} with respect to the contour $\Gamma_N$. Therefore, $|q_k - \tilde q_k| \le C \delta_1$ for $k = \overline{1,N-1}$. Consequently,
$$
\Bigg|1-\dfrac{Q(\la)}{\tilde Q(\la)}\Bigg| \le \dfrac{\Bigg|\sum\limits_{k=0}^{N-1}(q_k-\tilde q_k)\la^k\Bigg|}{\Bigg|\sum\limits_{k=0}^{N-1}\tilde q_k\la^k\Bigg|} \le C\delta_1
$$
for $\lambda$ on compact sets excluding eigenvalues $\{ \la_{n1} \}_{n = 1}^{N-1}$.

Next, using \eqref{r1_series} and taking Lemmas~\ref{lem:varphi} and \ref{quasider} into account, we estimate
\begin{gather*}
|\tilde r_1(\la) - r_1(\la)| \le |\tilde r_1(\la)|\Bigg|1 - \prod\limits_{k = 1}^{N-1} \dfrac{\la-\la_{k0}}{\la-\la_{k1}}\Bigg| + \\
\Bigg|1 - \prod\limits_{k = 1}^{N-1} \dfrac{\la-\la_{k0}}{\la-\la_{k1}}\Bigg|\Bigg|\dfrac{1}{2\pi i}\oint\limits_{\Gamma_N}{\dfrac{\tilde r_1(\la)\tilde\varphi^{[1]}(\pi, \mu) + \tilde r_2(\la)\tilde\varphi(\pi, \mu)}{\la - \mu}\varphi(\pi, \mu)\hat M_N(\mu)}d\mu\Bigg| + \\
\Bigg|\dfrac{1}{2\pi i}\oint\limits_{\Gamma_N}{\dfrac{\tilde r_1(\la)\tilde\varphi^{[1]}(\pi, \mu) + \tilde r_2(\la)\tilde\varphi(\pi, \mu)}{\la - \mu}\varphi(\pi, \mu)\hat M_N(\mu)}d\mu\Bigg| \le C\delta_1.
\end{gather*}

Using the previous statement, we get that $|\tilde r_1(\la) - r_1(\la)| \le C\delta_1$ for $\la$ on compact sets. Applying interpolation argument, we deduce the following estimate for the polynomial coefficients: $|c_n - \tilde c_n| \le C\delta_1$, $n = \overline{0, p}$. For $r_2(\la)$ the proof is similar.
\end{proof}

Thus, we have reconstructed the function $\sigma \in L_2(0,\pi)$ and the polynomials $r_1(\la)$ and $r_2(\la)$. Consider the problem $L = L(\sigma, r_1, r_2)$. Using \eqref{series_phi}, we find $\Phi(x, \la)$ and its quasi-derivative. Then, the following lemma can be proved by technical calculations.

\begin{lem} \label{solutions}
The following relations hold:
\begin{gather*}
-(\varphi^{[1]}(x, \la))' - \sigma(x)\varphi^{[1]}(x, \la) - \sigma^2(x)\varphi(x, \la) = \la\varphi(x, \la),\\
\varphi(0, \la) = 1, \quad \varphi^{[1]}(0, \la) = 0;\\
-(\Phi^{[1]}(x, \la))' - \sigma(x)\Phi^{[1]}(x, \la) - \sigma^2(x)\Phi(x, \la) = \la\Phi(x, \la),\\
\Phi^{[1]}(0, \la) = 1, \quad r_1(\la)\Phi^{[1]}(\pi, \la) + r_2(\la)\Phi(\pi, \la) = 0.
\end{gather*}

Moreover, $\{\la_n, \alpha_n\}_{n \ge 1}$ are the spectral data of the problem $L$.
\end{lem}

Lemmas \ref{est_sigma}, \ref{polynomials_est}, \ref{solutions} together prove Theorem~\ref{thm_mult}.

Proceed to a perturbation of an infinite number of simple eigenvalues.

\begin{thm}[\cite{ChitBondStab}]\label{stability_thm_m}
	Let $\tilde L = L(\tilde \sigma, \tilde r_1, \tilde r_2)$ be a fixed boundary value problem of form \eqref{eqv}-\eqref{bc} with $\tilde \sigma \in L_2(0,\pi)$ and $(\tilde r_1, \tilde r_2) \in \mathcal R_p$. Let the eigenvalues $\tilde \la_n$ with numbers $n \ge N$ be simple.
	Then, there exists $\delta_0 > 0$, depending on the problem $\tilde L$, such that, for any complex numbers $\{\la_n, \alpha_n\}_{n \ge 1}$ satisfying the conditions
	\begin{gather*}
	\la_n = \tilde \la_n, \quad \alpha_n = \tilde \alpha_n, \quad n < N, \\
	\delta_2 := \bigg(\sum\limits_{n=N}^\infty(|\tilde \rho_n - \rho_n|+|\tilde\alpha_n - \alpha_n|)^2\bigg)^{\frac{1}{2}} \le \delta_0,
	\end{gather*}
	there exist a complex-valued function $\sigma \in L_2(0, \pi)$ and polynomials $(r_1, r_2)\in\mathcal R_p$, which are the solution of Inverse Problem~\ref{ip:main3} for the data $\{\la_n, \alpha_n\}_{n \ge 1}$. Moreover,
	\begin{equation*} 
	\|\sigma - \tilde\sigma\|_{L_2(0, \pi)} \le C\delta_2, \quad |c_n - \tilde c_n| \le C\delta_2, \quad |d_n - \tilde d_n| \le C\delta_2, \quad n = \overline{0, p},
	\end{equation*}
	where the constant $C$ depends only on $\tilde L$.
\end{thm}

Now, Theorem~\ref{stability_thm} can be obtained by combining Theorems~\ref{thm_mult} and~\ref{stability_thm_m}. In order to show this, we will perturb the spectral data of the problem $\tilde L$ by parts. 
Let $\tilde L = L(\tilde \sigma, \tilde r_1, \tilde r_2)$ be a fixed problem satisfying the hypothesis of Theorem~\ref{thm_mult}.
First, we perturb the spectral data only for multiple eigenvalues. 
Consider any data $\{ \check \la_n, \check \alpha_n\}_{n \ge 1}$ satisfying the conditions of Theorem~\ref{thm_mult} with so small $\delta_1$ that the assertion of this theorem holds. That is,
$$
\check\la_n = \tilde\la_n, \quad \check\alpha_n = \tilde\alpha_n, \quad n \ge N,
$$
and $\{ \check\la_n, \check\alpha_n \}$  are arbitrarily perturbed values $\{ \tilde \la_n, \tilde \alpha_n\}$ for $n < N$, whose perturbation fulfills the estimate
$$
\delta_1 := \max_{\la \in \Gamma_N} |\check M_N(\la) - \tilde M_N(\la)| \le \delta_0.
$$
Then, $\{ \check\la_n, \check\alpha_n \}_{n \ge 1}$ are the spectral data of some problem $\check L = L(\check \sigma, \check r_1, \check r_2)$ and
\begin{equation} \label{estcheck}
\|\tilde\sigma(x) - \check\sigma(x)\|_{L_2(0,\pi)} \le C\delta_1, \quad |\tilde c_n - \check c_n| \le C\delta_1, \quad |\tilde c_n - \check c_n| \le C\delta_1, \quad n = \overline{0, p},
\end{equation}
where $C$ depends only on $\tilde L$.

Next, we perturb the infinite part of the spectral data, which contains only simple eigenvalues. Let us apply Theorem~\ref{stability_thm_m} to the problem $\check L$ instead of $\tilde L$ and to arbitrary data $\{ \la_n, \alpha_n \}_{n \ge 1}$ satisfying the conditions
\begin{gather} \nonumber
\la_n = \check \la_n, \quad \alpha_n = \check \alpha_n, \quad n < N, \\ \label{checkrho}
\delta_2 := \left( \sum_{n = N}^{\infty} (|\check \rho_n - \rho_n| + |\check \alpha_n - \alpha_n|)^2\right)^{\frac{1}{2}} \le \delta_0
\end{gather}
for sufficiently small $\delta_0$. Then, $\{ \la_n, \alpha_n \}_{n \ge 1}$ are the spectral data of some problem $L = L(\sigma, r_1, r_2)$ and
\begin{equation} \label{Lest}
\|\check\sigma(x) - \sigma(x)\|_{L_2(0,\pi)} \le C\delta_2, \quad |\check c_n - c_n| \le C\delta_2, \quad |\check c_n - c_n| \le C\delta_2, \quad n = \overline{0, p}.
\end{equation}
Note that the choice of $\delta_0$ in \eqref{checkrho} and of the constant $C$ in \eqref{Lest} depends on $\check L$. However, it follows from the proof of Theorem~\ref{stability_thm_m} in \cite{ChitBondStab} that one can choose fixed $\delta_0$ and $C$ for any $\check L$, whose coefficients are sufficiently close to the ones of $\tilde L$ in the sense of the estimate \eqref{estcheck}.
So, we get from \eqref{estcheck} and \eqref{Lest} that
$$
\|\tilde\sigma(x) - \sigma(x)\|_{L_2(0,\pi)} \le C\delta, \quad |\tilde c_n - c_n| \le C\delta, \quad |\tilde c_n - c_n| \le C\delta, \quad n = \overline{0, p},
$$
where $\delta = \delta_1 + \delta_2$ and $C$ depends only on $\tilde L$. This 
yields the claim of Theorem~\ref{stability_thm}.

\section{Generalized Cauchy data} \label{sec:cauchy}

In this section, we provide the proof of Theorem~\ref{cauchy_thm} on local solvability and stability of Inverse Problem~\ref{ip:c}. In order to prove this theorem, we use reduce of the inverse problem by the generalized Cauchy data to the one by the spectral data $\{ \hat M_N(\la) \}_{\la \in \Gamma_N} \cup \{ \la_n, \alpha_n\}_{n \ge N}$. After that, we can apply our results for Inverse Problem~\ref{ip:main3} which are presented in Theorem~\ref{stability_thm}.

Using \eqref{delta_0}-\eqref{delta_1} and Rouche's Theorem, we prove the following lemma:

\begin{lem} \label{lem:asympt}
Let $G(t), J(t)$ be arbitrary functions of $L_2(0, \pi)$, $C_0, \dots, C_{p}, D_0, \dots, D_{p-1} \in \mathbb C$ be arbitrary complex numbers. Let the functions $\Delta_0(\la)$ and $\Delta_1(\la)$ be constructed by formulae \eqref{delta_0} and \eqref{delta_1}, respectively. Then, for $j = 0, 1$, the function $\Delta_j(\la)$ has the set of zeros $\{\theta_{nj}\}_{n \ge 1}$, which are numbered so that $|\theta_{nj}| \le |\theta_{n+1,j}|$ and have the asymptotic behavior
$$
\sqrt{\theta_{nj}} = n - p - \dfrac{j + 1}{2} + \varkappa_{nj}, \quad \{\varkappa_{nj}\} \in l_2.
$$
\end{lem}

Let $\{ G(t), J(t), C_0, \dots, C_p, D_0, \dots, D_{p-1} \}$ be arbitrary data satisfying the hypothesis of Lemma~\ref{lem:asympt} (not necessarily being the generalized Cauchy data of a certain boundary value problem $L$). Then, we can construct $\Delta_0(\la)$ and $\Delta_1(\la)$ by \eqref{delta_0} and \eqref{delta_1}, respectively, and define 
$$
M(\la) := -\dfrac{\Delta_0(\la)}{\Delta_1(\la)}.
$$

Next, for brevity, put $\theta_n := \theta_{n1}$ for $n \ge 1$.
Introduce the following values similarly to weight numbers:
$$
\alpha_{n+k} := \mathop{\mathrm{Res}}_{\la=\theta_n} (\la - \theta_n)^k M(\la), \quad n \in I, \quad k = \overline{0, m_n-1},
$$
where $m_n$ is the multiplicity of $\theta_n$ and $I$ is the set of the indexes of the distinct values among $\{ \theta_n \}_{n \ge 1}$. So, we can define the index $N$ and the contour $\Gamma_N$ for the data $\{ \tilde G(t), \tilde J(t), \tilde C_0, \dots, \tilde C_p, \tilde D_0, \dots, \tilde D_{p-1} \}$ according to Notation~\ref{not:GammaN}. Using these definitions, we formulate the following lemma.

\begin{lem} \label{lem:Cauchy}
Let $\tilde G(t)$, $\tilde J(t) \in L_2(0, \pi)$ be fixed complex-valued functions and $\tilde C_0, \dots, \tilde C_{p}, \tilde D_0, \dots, \tilde D_{p-1}$ be fixed complex numbers. Then, there exists $\varepsilon > 0$ such that, for all $G(t), J(t) \in L_2(0, \pi)$, $C_0, \dots, C_{p}, D_0, \dots, D_{p-1} \in \mathbb C$, satisfying the condition \eqref{difCauchy} from Theorem~\ref{cauchy_thm}, the points $\{\theta_n\}_{n=1}^{N-1}$ are located inside $\Gamma_N$, $m_n = 1$ for $n \ge N$, and
\begin{gather} \label{hatM}
	\max\limits_{\la\in\Gamma_N}|\tilde M(\la) - M(\la)| \le C\delta, \\ \label{sumde} 
\Big(\sum\limits_{n=N}^{\infty}(|\sqrt{\theta_n} - \sqrt{\tilde\theta_n}| + |\alpha_n - \tilde\alpha_n|)^2\Big)^{\frac{1}{2}}\le C\delta.
\end{gather}
\end{lem}

\begin{proof}
Let $\{ \tilde G(t), \tilde J(t), \tilde C_0, \dots, \tilde C_p, \tilde D_0, \dots, \tilde D_{p-1} \}$ satisfy the hypothesis of the lemma, and let $\{ G(t), J(t), C_0, \dots, C_p, D_0, \dots, D_{p-1} \}$ fulfill the inequality \eqref{difCauchy} with some $\varepsilon > 0$. Throughout this proof, the constants $C$ in estimates depend only on $\{ \tilde G(t), \tilde J(t), \tilde C_0, \dots, \tilde C_p, \tilde D_0, \dots, \tilde D_{p-1} \}$ and $\varepsilon$.
We divide the proof into three steps.

\smallskip

\textit{Step 1. Behavior of $\{ \theta_n\}_{n = 1}^{N-1}$}. Clearly, the functions $\Delta_0(\la)$ and $\Delta_1(\la)$ defined via \eqref{delta_0}--\eqref{delta_1} are entire in $\la$. So, relying on Notation~\ref{not:GammaN} and on the relations \eqref{delta_0} and \eqref{difCauchy}, we get the estimates
\begin{equation} \label{estD1}
|\tilde\Delta_1(\la)| \ge C, \quad |\tilde \Delta_1(\la) - \Delta_1(\la)| \le C \delta, \quad \la\in\Gamma_N.
\end{equation}

Hence, if $\varepsilon > 0$ is sufficiently small, then
$$
\dfrac{|\tilde\Delta_1(\la) - \Delta_1(\la)|}{|\tilde\Delta_1(\la)|} \le C\delta \le C\varepsilon < 1, \quad \la\in\Gamma_N.
$$
Below in this proof, we assume that $\varepsilon$ is so small that all the proved estimates hold. Consequently, by Rouche's Theorem, $\Delta_1(\la)$ has the same number of zeros (counting with multiplicities) inside $\Gamma_N$ as $\tilde\Delta_1(\la)$ has. Obviously, these zeros are $\{\theta_n\}_{n=1}^{N-1}$.

It can be also shown that
\begin{equation} \label{estD2}
\dfrac{|\tilde\Delta_0(\la) - \Delta_0(\la)|}{|\tilde\Delta_1(\la)|}\le C\delta, \quad |\Delta_1(\la)| \ge C, \quad  \la \in \Gamma_N.
\end{equation}
Using the estimates \eqref{estD1} and \eqref{estD2}, we obtain
$$
|\tilde M(\la) - M(\la)| = \Bigg|\dfrac{\Delta_0(\la)}{\Delta_1(\la)} - \dfrac{\tilde\Delta_0(\la)}{\tilde\Delta_1(\la)}\Bigg| 
 \le C\delta, \quad \la \in \Gamma_N.
$$

\smallskip

\textit{Step 2. Behavior of $\{ \theta_n \}_{n \ge N}$}.
Put $\tilde \nu_n = \sqrt{\tilde\theta_n}$, $\arg \tilde\nu_n \in (-\tfrac{\pi}{2}, \tfrac{\pi}{2}]$.
For $n \ge N$, consider the contours $\gamma_n = \{\rho \in \mathbb C\colon |\rho - \tilde\nu_n| = r\}$, where $r \le \dfrac{|\tilde\nu_n - \tilde\nu_{n+1}|}{2}$ is fixed.

Let $\rho_0 \in \gamma_n$, so $\rho_0 = \tilde\nu_n + re^{i\varphi}$, where $\varphi \in [0, 2\pi]$. In view of Lemma~\ref{lem:asympt} 
$$
\tilde\nu_n = n\big(1 + O(n^{-1})\big), \quad n \ge N.
$$
It can be shown that $|\sin\rho_0t| \le C$. Using the latter estimate together with \eqref{delta_1} and \eqref{difCauchy}, we derive
\begin{gather*}
|\tilde\Delta_1(\rho^2) - \Delta_1(\rho^2)| \le \big|\rho^{2p+1}\big|\Bigg|\int\limits_{0}^{\pi} {(\tilde G(t) - G(t))\sin\rho t}dt \Bigg| + \sum\limits_{n=0}^{p}|\tilde C_n - C_n|\big|\rho^{2n}\big| \le \\
Cn^{2p+1}\|\tilde G(t) - G(t)\|_{L_2(0, \pi)} + Cn^{2p}\delta \le Cn^{2p+1}\delta, \quad \rho \int \mbox{int}\, \gamma_n.
\end{gather*}

Let us introduce the contour $\gamma^1_n = \{\rho \in \mathbb C\colon |\rho - \tilde \nu_n| = r_1\}$, where $r_1 < r$.
Using Taylor's Theorem, we can get for $\rho \in \mbox{int}\,\gamma^1_n$ that
$$
\tilde\Delta_1(\rho^2) = \tilde\Delta_1(\tilde\nu_n^2) + \dfrac{d}{dz}\tilde\Delta_1(z^2)\Big|_{z=\tilde\nu_n}(\rho-\tilde\nu_n) + R_n(\rho),
$$
where
$$
R_n(\rho) = \dfrac{1}{2\pi i}(\rho - \tilde\nu_n)^2\oint\limits_{\gamma_n} {\dfrac{\tilde\Delta_1(z^2)}{(z-\tilde\nu_n)^2(z - \rho)}}dz.
$$

Using direct calculations, we obtain the estimate 
$$
\Bigg| \dfrac{1}{2\pi i}\oint\limits_{\gamma_n} {\dfrac{\tilde\Delta_1(z^2)}{(z-\tilde\nu_n)^2(z - \rho)}}dz \Bigg| \le Cn^{2p+1}, \quad \rho \in \mbox{int} \, \gamma_n^1,
$$
and, consequently,
$$
R_n(\rho) = O(n^{2p+1}(\rho - \tilde\nu_n)^2), \quad \rho \in \mbox{int}\,\gamma_n^1.
$$

Then, we get:
\begin{align} \notag
\tilde\Delta_1(\nu_n^2) - \Delta_1(\nu_n^2) &= \nu_n^{2p+1}\int\limits_{0}^{\pi} {(\tilde G(t) - G(t))\sin\nu_n t}dt + \sum\limits_{k=0}^{p}(\tilde C_k - C_k)\nu_n^{2n} = \\ \label{for_diff_nu}
&\dfrac{d}{dz}\tilde\Delta_1(z^2)\Bigg|_{z=\tilde\nu_n}(\nu_n-\tilde\nu_n) + O(n^{2p+1}(\nu_n - \tilde\nu_n)^2), \quad n \to \infty.
\end{align}

Using direct calculations, we obtain the estimates
\begin{gather*}
\Bigg|\dfrac{d}{dz}\tilde\Delta_1(z^2)\Big|_{z=\tilde\nu_n}\Bigg| \ge Cn^{2p+1}, \\
\Bigg|\int\limits_{0}^{\pi} {(\tilde G(t) - G(t))\sin\nu_n t}dt\Bigg| \le C|\hat G_n| + C\delta\varkappa_n,
\end{gather*}
where $\hat G_n = \int\limits_{0}^{\pi}\hat G(t)\sin nt dt$, $\{ \varkappa_n \} \in l_2$, and $\| \{ \varkappa_n\} \|_{l_2} \le C$.
So, we get from \eqref{for_diff_nu}:
$$
|\nu_n - \tilde\nu_n| \le C|\hat G_n| + C\delta\varkappa_n + \dfrac{C\delta}{n} + \dots + \dfrac{C\delta}{n^{2p-1}}, \quad n \ge N.
$$

Using Bessel's inequality, we obtain
$$
\Big(\sum\limits_{n=N}^{\infty} |\hat G_n|^2\Big)^{\frac{1}{2}} \le C\delta.
$$

Hence
\begin{gather}\label{neq_nu}
\Big(\sum\limits_{n=N}^{\infty} |\nu_n - \tilde\nu_n|^2\Big)^{\frac{1}{2}} \le C\delta.
\end{gather}

\smallskip 

\textit{Step 3. Behavior of $\{ \alpha_n \}_{n \ge N}$}.
Next, we have that $\{\tilde\theta_n\}_{n\ge N}$ are simple poles of $\tilde M(\la)$, then
$$
\tilde\alpha_n = -\dfrac{\tilde\Delta_0(\tilde\theta_n)}{\tilde\Delta_1^{\langle 1 \rangle}(\tilde\theta_n)}, \quad n \ge N.
$$

Consequently,
\begin{gather} \label{alpha_diff}
\alpha_n - \tilde\alpha_n = \dfrac{\Delta_1^{\langle 1 \rangle}(\theta_n)(\tilde\Delta_0(\tilde\theta_n) - \Delta_0(\theta_n)) + \Delta_0(\theta_n)(\Delta_1^{\langle 1 \rangle}(\theta_n) - \tilde\Delta_1^{\langle 1 \rangle}(\tilde\theta_n))}{\tilde\Delta_1^{\langle 1 \rangle}(\tilde\theta_n)\Delta_1^{\langle 1 \rangle}(\theta_n)}.
\end{gather}

Using direct calculations, we can get the following estimates:
\begin{gather} \label{cos_diff}
|\cos\tilde\nu_n\pi - \cos\nu_n\pi| \le C|\tilde\nu_n - \nu_n|, \\ \label{sin_diff}
|\sin\tilde\nu_n\pi - \sin\nu_n\pi| \le C|\tilde\nu_n - \nu_n|.
\end{gather}

From \eqref{delta_1}, we obtain
\begin{align} \nonumber
\tilde\Delta_1^{\langle 1 \rangle}(\tilde\theta_n) - \Delta_1^{\langle 1 \rangle}(\theta_n) &= -\dfrac{\pi}{2}(\tilde\nu_n^{2p}\cos\tilde\nu_n\pi - \nu_n^{2p}\cos\nu_n\pi) + \\ \nonumber
&\dfrac{1}{2}\Big(\tilde\nu_n^{2p}\int\limits_{0}^{\pi} t\tilde G(t)\cos\tilde\nu_n t \, dt - \nu_n^{2p}\int\limits_{0}^{\pi} tG(t)\cos\nu_n t \, dt\Big) - \\ \nonumber
& \Big(p + \dfrac{1}{2}\Big)(\tilde\nu_n^{2p-1}\sin\tilde\nu_n\pi - \nu_n^{2p-1}\sin\nu_n\pi) + \\ \nonumber
& \Big(p + \dfrac{1}{2}\Big)\Big(\tilde\nu_n^{2p-1}\int\limits_{0}^{\pi} \tilde G(t)\sin\tilde\nu_n t \, dt - \nu_n^{2p-1}\int\limits_{0}^{\pi} G(t)\sin\nu_n t \, dt\Big) + \\ \label{difDeltap}
& \sum\limits_{k=0}^{p-1}(p-k)(\tilde C_k\tilde\nu_n^{2(k-1)} - C_k\nu_n^{2(k-1)}).
\end{align}

Using \eqref{difCauchy}, \eqref{cos_diff}, \eqref{sin_diff}, and \eqref{difDeltap}, we deduce
\begin{gather}\label{der_delta_1_diff}
|\tilde\Delta_1^{\langle 1 \rangle}(\tilde\theta_n) - \Delta_1^{\langle 1 \rangle}(\theta_n)| \le Cn^{2p}|\tilde\nu_n - \nu_n| + Cn^{2p}|\hat L_n| + Cn^{2p-1}|\hat P_n| + Cn^{2p - 2}\delta,
\end{gather}
where $\hat L_n := \int\limits_{0}^{\pi} t\hat G(t)\cos\tilde\nu_n t \, dt$, $\hat P_n := \int\limits_{0}^{\pi} \hat G(t)\sin\tilde\nu_n t \, dt$.

In the same way, we can get
\begin{gather}\label{delta_0_diff}
|\tilde\Delta_0(\tilde\theta_n) - \Delta_0(\theta_n)| \le Cn^{2p}|\tilde\nu_n - \nu_n| + Cn^{2p}|\hat S_n| + Cn^{2p - 2}\delta,
\end{gather}
where $\hat S_n = \int\limits_{0}^{\pi} \hat J(t)\cos\tilde\nu_n t \, dt$.

Substituting \eqref{der_delta_1_diff}-\eqref{delta_0_diff} into \eqref{alpha_diff}, we obtain
$$
|\alpha_n - \tilde\alpha_n| \le C|\tilde\nu_n - \nu_n| + C|\hat S_n| + C|\hat L_n| + \dfrac{C|\hat P_n|}{n} + \dfrac{C\delta}{n^2}, \quad n \ge N.
$$

According to Bessel's inequalities for $\hat S_n$, $\hat L_n$, and $\hat P_n$, and the inequality \eqref{neq_nu}, we get the estimate
\begin{gather*}
\Big(\sum\limits_{n=N}^{\infty} |\alpha_n - \tilde\alpha_n|^2\Big)^{\frac{1}{2}} \le C\delta,
\end{gather*}
which concludes the proof.
\end{proof}

Using technical calculations, one can easily show that the estimates \eqref{hatM} and \eqref{sumde} together imply $\max\limits_{\la \in \Gamma_N} |\hat M_N(\la)| \le C \delta$.
Consequently, Lemma~\ref{lem:Cauchy} and Theorem~\ref{stability_thm} immediately yield Theorem~\ref{cauchy_thm}.

\medskip

\textbf{Funding}: This work was supported by Grant 21-71-10001 of the Russian Science Foundation, https://rscf.ru/en/project/21-71-10001/.

% Я это пока закомментировала, потому что для arxiv это не нужно.

%\medskip

%\textbf{Competing interests}: The paper has no conflict of interests.

%\medskip

%\textbf{Data Availability}: Data sharing not applicable to this article as no datasets were generated or analysed during the current study

\noindent Egor Evgenevich Chitorkin \\
1. Institute of IT and Cybernetics, Samara National Research University, \\
Moskovskoye Shosse 34, Samara 443086, Russia, \\
2. Department of Mechanics and Mathematics, Saratov State University, \\
Astrakhanskaya 83, Saratov 410012, Russia, \\
e-mail: {\it chitorkin.ee@ssau.ru} \\

\noindent Natalia Pavlovna Bondarenko \\
1. Department of Mechanics and Mathematics, Saratov State University, \\
Astrakhanskaya 83, Saratov 410012, Russia, \\
2. Department of Applied Mathematics and Physics, Samara National Research University, \\
Moskovskoye Shosse 34, Samara 443086, Russia, \\
3. Peoples' Friendship University of Russia (RUDN University), \\
6 Miklukho-Maklaya Street, Moscow, 117198, Russia, \\
e-mail: {\it bondarenkonp@info.sgu.ru}
\end{document}